\documentclass[dvipsnames,onefignum,onetabnum,final]{siamart250211}

\title{Computing accurate singular values using
  a mixed-precision one-sided Jacobi algorithm}

\author{Zhengbo~Zhou\footnote{C\MakeLowercase{orresponding author
(\email{zhengbo.zhou@postgrad.manchester.ac.uk}).}} \thanks{Department of Mathematics,
    University of Manchester,
    Manchester M13 9PL,
    United~Kingdom. \\ \textbf{Funding:} The first author was supported by  the University of Manchester Research
Scholar Award and the second author was supported by Engineering and Physical Sciences Research Council grant EP/W018101/1.} \and Fran\c{c}oise~Tisseur\footnotemark[2]
  \and Marcus~Webb\footnotemark[2]}

\headers{A mixed-precision one-sided Jacobi algorithm}{
  Zhengbo~Zhou, Fran\c{c}oise~Tisseur and Marcus~Webb}

\usepackage{amsfonts}
\usepackage{amsmath}
\usepackage{amssymb} 
\usepackage{bookmark}
\usepackage{booktabs}
\usepackage{enumitem}
\usepackage{graphics,graphicx}
\usepackage{mathtools}
\usepackage[utf8]{inputenc}
\usepackage{upref}
\usepackage{subdepth}
\usepackage[normalem]{ulem}
\usepackage{comment}
\usepackage[caption=false]{subfig}

\usepackage{tcolorbox}
\tcbuselibrary{listings,breakable}
\colorlet{LightCornflowerBlue}{CornflowerBlue!50}
\newtcolorbox{mybox}{
  colback=LightCornflowerBlue,
  colframe=black,
  boxrule=0pt,
  breakable,
  arc=0pt,
  outer arc=0pt
}

\makeatletter
\g@addto@macro\bfseries{\boldmath}
\makeatother

\usepackage{pgfplots,tikz}
\pgfplotsset{
  width=0.35\textwidth,
  height=0.28\textwidth,
  scale only axis,
  xlabel near ticks,
  ylabel near ticks,
  every axis title shift=3pt,
  grid=both,
}
\usetikzlibrary{plotmarks}
\pgfplotsset{compat=1.18}

\usepackage{algorithm}
\usepackage{algorithmicx}
\usepackage{algpseudocode}

\algrenewcommand\algorithmiccomment[1]{\hfill\textcolor{red}{\%\ #1}}
\usepackage{chngcntr}
\counterwithout{algorithm}{section}

\newsiamthm{assumption}{Assumption}
\newsiamremark{remark}{Remark}

\mathchardef\Gamma="7100
\mathchardef\Delta="7101
\mathchardef\Theta="7102
\mathchardef\Lambda="7103
\mathchardef\Xi="7104
\mathchardef\Pi="7105
\mathchardef\Sigma="7106
\mathchardef\Upsilon="7107
\mathchardef\Phi="7108
\mathchardef\Psi="7109
\mathchardef\Omega="710A

\newcommand{\R}{\mathbb{R}}

\newcommand{\wh}{\widehat}
\newcommand{\wt}{\widetilde}

\newcommand{\mn}{^{m\times n}}

\newcommand{\nn}{^{n\times n}}
\newcommand{\tp}{^{T}}

\newcommand{\inv}{^{-1}}
\newcommand{\pinv}{^{\dagger}}
\newcommand{\diag}{\operatorname{diag}}

\newcommand{\abs}[1]{\lvert{#1}\rvert}

\newcommand{\norm}[1]{\|{#1}\|}
\newcommand{\fnorm}[1]{\norm{#1}_F}
\newcommand{\tnorm}[1]{\norm{#1}_2}

\newcommand{\mnorm}[1]{\norm{#1}_{\mathrm{max}}}

\newcommand{\off}{{\operatorname{off}}}
\DeclareMathOperator{\fl}{\operatorname{f\kern.2ptl}}

\newcommand{\iter}[1]{^{(#1)}}
\newcommand{\tol}{{\mathrm{tol}}}
\newcommand{\eps}{{\varepsilon}}

\newcommand{\ul}{{u_{\ell}}}
\newcommand{\uh}{{u_{h}}}

\usepackage{listings}
\definecolor{gray}{rgb}{0.5,0.5,0.5}
\definecolor{mauve}{rgb}{0.58,0,0.82}
\definecolor{lightgrey}{rgb}{0.9,0.9,0.9}
\definecolor{darkgreen}{rgb}{0,0.6,0}
\lstdefinestyle{mystyle}{%
  language=matlab,
  morekeywords={anymatrix},
  showstringspaces=false,
  columns=flexible,
  keepspaces = true,
  basicstyle={\small\ttfamily},
  numbers=none,
  numberstyle=\tiny\color{gray},
  keywordstyle=\color{blue},
  commentstyle=\color{darkgreen},
  stringstyle=\color{mauve},
  breakatwhitespace=true,
  upquote=true,
  xleftmargin=5.0ex
}
\def\inline{\lstinline[basicstyle=\upshape\ttfamily]}
\usepackage[T1]{fontenc}
\LoadFontDefinitionFile{\encodingdefault}{cmtt}
\ExpandArgs{c}\def{\encodingdefault/cmtt/m/it}{<->ssub*cmtt/m/sl}

\def\2nsit{2--step NS iteration}

\def\At{{\wt{A}}}

\def\Dt{{\wt{D}}}

\newcommand{\Atcomp}{{\wt{A}_{\mathrm{comp}}}}
\newcommand{\Athcomp}{{\wt{A}_{\mathrm{h,comp}}}}

\def\ferrk{{\eps_{fwd}\iter{k}}}

\def\at{\wt{a}}

\def\gammah{\gamma_{h}}
\def\scond{\kappa_{2}^{D}}
\def\k{\kappa_{2}}

\def\Vt{\wt{V}}

\def\obliq{\operatorname{obliq}}

\begin{document}
\maketitle

\begin{abstract}
We present a relative forward error analysis of a mixed-precision preconditioned one-sided Jacobi algorithm, analogous to a two-sided version introduced in [N.~J.~Higham, F.~Tisseur, M.~Webb and Z.~Zhou, \textit{SIAM J. Matrix Anal. Appl.} 46 (2025), pp.~2423–2448], which uses low precision to compute the preconditioner, applies it in high precision, and computes the singular value decomposition using the one-sided Jacobi algorithm at working precision. 
Our analysis yields relative forward error bounds {involving the scaled condition number of the preconditioned matrix rather than that of the input matrix, potentially improving accuracy when the former is smaller.}
We present and analyze two approaches for constructing effective preconditioners. 
Our numerical experiments demonstrate that our algorithm achieves smaller relative forward errors than the LAPACK routines \texttt{DGESVJ} and \texttt{DGEJSV}, as well as the MATLAB function \texttt{svd}, particularly for ill-conditioned matrices. 
Timing tests show that our approach accelerates the convergence of the Jacobi iterations and that the dominant cost arises from a single high-precision matrix–matrix multiplication. With improved software or hardware support for this bottleneck, our algorithm would be faster than the LAPACK one-sided Jacobi algorithm \texttt{DGESVJ} and comparable in speed to the state-of-the-art preconditioned one-sided Jacobi algorithm \texttt{DGEJSV}, but much more accurate.
\end{abstract}



\section{Introduction}\label{sec.Intro}
The one-sided Jacobi algorithm is an iterative method for computing the singular value decomposition (SVD) of a matrix~\cite{nash75}, \cite[p.~492]{gova13-MC4}. It is closely related to the Jacobi eigenvalue algorithm, which computes the eigendecomposition of a symmetric matrix \cite{jaco46}, \cite[p.~476]{gova13-MC4} --- the one-sided Jacobi algorithm applied to $A$ is mathematically equivalent to the Jacobi eigenvalue algorithm applied to $A\tp A$. We describe the Jacobi eigenvalue algorithm as ``two-sided''.

Similarly to its two-sided counterpart, the one-sided Jacobi algorithm 
has been shown by Demmel and Veseli\'{c}~\cite{deve92} to be more accurate than algorithms based on an initial bidiagonal reduction. For {a full column rank matrix} $A \in \R^{m\times n}$, define the condition number $\k(A) = \sigma_{\max}(A)/\sigma_{\min}(A)$, and the \textit{one-sided scaled condition number},
\begin{equation}\notag
  \scond(A) = \k(AD), \qquad
  D = \diag(\tnorm{a_i}^{-1}),
\end{equation}
where $a_i$ is the $i$th column of $A$. Demmel and Veseli\'{c} show that 
if the one-sided Jacobi algorithm terminates with the stopping criterion
\begin{equation}\label{eq.one-sided-stopping-critertion}
  \text{stop when} \quad 
  \abs{a_i \tp a_j}
  \le \tol \cdot \tnorm{a_i}\tnorm{a_j}
  \quad \text{for all $i \neq j$,}
\end{equation}
where $\tol$ is a user-specified tolerance {chosen to be of order $u$}, then the $k$th largest computed singular value $\wh{\sigma}_k$ satisfies
\begin{equation}\label{eqn:demmelveselicbound}
  \frac{\abs{\wh{\sigma}_{k}(A) - \sigma_{k}(A)}}{\sigma_{k}(A)} 
  \le p(m,n)\, u\, \scond(A),
\end{equation}
where $\sigma_{k}(A)$ is the $k$th largest exact singular value of $A$, 
$p(m,n)$ is a low-degree polynomial, and 
$u$ is the unit roundoff of the working precision. For algorithms based on an initial bidiagonal reduction, a similar bound can be derived, but with $\scond(A)$ replaced by the standard condition number $\k(A)$ \cite[p.~249]{demm97-ANLA}. Thus, the one-sided Jacobi algorithm is likely to be more accurate when $\scond(A) \ll \k(A)$.

The present work follows a recent paper on a mixed-precision preconditioned two-sided Jacobi algorithm for symmetric positive definite $A$ \cite{htwz25} (which follows earlier work in \cite{zhou22-MPJA}), where, in the positive definite case, 
relative forward error bounds for the computed eigenvalues are derived in terms of $\kappa_2^S(\At)$, a two-sided  form of the scaled condition number,
where $\At$ is the matrix obtained after a preconditioning step (which guarantees $\kappa_2^S(\At) \ll \kappa_2^S(A)$).

In this paper we adapt the algorithms and analyses in~\cite{htwz25} to the one-sided Jacobi algorithm applied to a full-rank matrix $A \in \R^{m\times n}$ with $m \geq n$. 
Let us emphasize immediately that this adaptation is nontrivial. Our first contribution is to derive mixed-precision preconditioned one-sided Jacobi algorithms, including a QR factorization step for dimension reduction, and to provide a full analysis of the relative forward error of the computed singular values. Our second contribution is to describe two approaches that use low-precision arithmetic to compute a nearly orthogonal preconditioner matrix $\Vt$ such that $\At = A\Vt$ has almost-orthogonal columns and a modest scaled condition number. The main bound obtained for the relative forward error is analogous to \eqref{eqn:demmelveselicbound}, but with $\scond(A)$ replaced by $\scond(\At)$, which we can guarantee to be modest even if $\scond(A)$ is not. 
The analysis in \cite{htwz25}, which bounds the two-sided condition number of $\At$ using the $\off$ quantity, does not directly translate to the one-sided case, so a new analysis involving a novel quantity we call the obliquity is developed to bound $\scond(\At)$ (see section \ref{sec.prec-and-svd}).
{In~\cite{ztw26b}, we analyze the accuracy of the computed eigenvectors and singular vectors of the mixed-precision Jacobi algorithms introduced in~\cite{htwz25} and the present paper.}

Recently, Gao et al.~\cite{gms25} and Zhang and Bai~\cite{zhba22a} introduced a 2-precision mixed-precision one-sided Jacobi algorithm. The present work, however, is on a 3-precision variant that obtains much more accurate singular values. Gao et al.~discuss the use of a QR factorization step \emph{before} the preconditioning step, whereas in the present work we must perform the QR factorization afterwards in order to achieve high accuracy (see section \ref{sec.qr}).

We performed several numerical experiments to {assess the relative forward accuracy of the algorithms} across a wide range of random and non-random matrices of varying sizes and condition numbers.
We have made our MATLAB code available at \url{https://github.com/zhengbo0503/Code_twz26a}. Timing tests showed that our approach accelerates the convergence of the Jacobi algorithm and that the dominant cost arises from a single high-precision matrix–matrix multiplication. With improved software or hardware support for this bottleneck, our algorithm would be faster than the LAPACK one-sided Jacobi algorithm \inline{DGESVJ} and comparable in speed to the state-of-the-art preconditioned one-sided Jacobi algorithm \inline{xGEJSV}, but much more accurate.

In section~\ref{sec.mp-one-sided-Jacobi}, we present Algorithm~\ref{alg.mp-precond-onesided-Jacobi}, the mixed-precision one-sided Jacobi algorithm, and state and prove the main result of this paper, Theorem~\ref{thm.main}. In section~\ref{sec.constr-of-prec}, we describe two approaches for constructing a suitable preconditioner for our algorithm, both adapted from~\cite{htwz25}. In section~\ref{sec.qr}, we discuss the use of an initial QR factorization in our algorithm, including whether the QR factorization should be performed before or after preconditioning. Numerical experiment results are presented in section~\ref{sec.numer-experi} to {assess the accuracy and performance of the algorithms in practice}. Conclusions are drawn in section~\ref{sec.conclusion}.

\section{Mixed-precision preconditioned one-sided Jacobi algorithm}
\label{sec.mp-one-sided-Jacobi}

In the remainder of the paper, $A \in \R^{m\times n}$ is a full-rank matrix and $m \geq n$. 
For any index $i$, we write $p_{i} = p_{i}(m,n)$ to mean a low-degree polynomial in $n$ and $m$. We use $u$ and $u_{h}$ to denote the unit roundoff of the working and the high precision, respectively.
The phrase ``at precision $u$'' means that an operation has been performed in floating point arithmetic with precision that has unit roundoff $u$.
{For the analysis that follows, we assume that all the precisions are IEEE binary formats to ensure promotion of data from lower to higher precision is exact~\cite[sect.~5.4.2]{ieee19}, and that neither overflow nor underflow occurs.}

Throughout sections \ref{sec.mp-one-sided-Jacobi} and \ref{sec.qr} we assume that a preconditioner $\wt{V} \in \R^{n\times n}$ can be computed such that it is numerically orthogonal
at working precision $u$, in that
\begin{equation}\label{prec-assumpt-orthogonality}
  \tnorm{\wt{V}\tp \wt{V} - I} \le p_{1} u < 1.
\end{equation}
We will discuss how to obtain such a preconditioner, using a low precision $u_\ell$, in more detail
in section~\ref{sec.constr-of-prec}.

Before entering into the analysis,
let us first describe
the \textit{mixed-precision preconditioned one-sided Jacobi algorithm}
in Algorithm~\ref{alg.mp-precond-onesided-Jacobi}. It is related to
the mixed-precision preconditioned
(two-sided) Jacobi algorithm introduced in~\cite[Alg.~1]{htwz25}.

\begin{algorithm}[tbhp!]
\caption{Mixed-precision preconditioned one-sided Jacobi algorithm}
\label{alg.mp-precond-onesided-Jacobi}
\begin{algorithmic}[1]
\Require{A full-rank matrix $A \in \R\mn$ with $m \geq n$,
  two precisions $u$ and $\uh$ with $0 < \uh < u$, and 
  a preconditioner $\Vt$ satisfying
  $\tnorm{\Vt\tp\Vt - I} \le p_1u < 1$.}
\Ensure{A computed SVD $U\Sigma V\tp$ of $A$.}
\Statex{}
\State{\label{li.compute-Atcomp}%
  Compute the preconditioned matrix $\At$
  by computing the product $A\Vt$ at precision $\uh$,
  which gives $\Athcomp$,
  and then demoting it to precision $u$, which yields $\Atcomp$.}
\State{\label{li.compute-SVD}%
  Compute an SVD, $\Atcomp = U\Sigma V_J\tp$
  using the one-sided Jacobi algorithm with
  stopping criterion~\eqref{eq.one-sided-stopping-critertion}
  at precision $u$.}
\State{\label{li.compute-right-svecs}%
  Construct the right singular vector matrix
  $V = \Vt V_J$ at precision $u$.}
\end{algorithmic}
\end{algorithm}

\subsection{The main theorem}
\label{sec.main-theorem}
Following~\cite{high02-ASNA2}, write
\begin{equation}\label{eq.gamma-gammah}
  \gammah = \frac{nu_{h}}{1-nu_{h}}
  < \gamma = \frac{nu}{1-nu} < 1.
\end{equation}
To investigate the behavior of the relative forward error
$\ferrk$ of Algorithm~\ref{alg.mp-precond-onesided-Jacobi},
we need the following assumptions. 
\begin{assumption}
\label{ass.main}
Let $u$, $u_{h}$, $\gamma$, and $\gammah$ be as in \eqref{eq.gamma-gammah},
and let $A, \At\in\R^{m\times n}$ and $p_{1}$ be as in
Algorithm~\ref{alg.mp-precond-onesided-Jacobi}.
We assume
\begin{enumerate}[label=\textup{(A\arabic*)}]
\item\label{it.6nukA}
$6nu(1-p_{1}u)\inv\k(A)<1$,
\item\label{it.gammah<u}
$\gammah <  \frac{(1-p_1u)u}{4(1+p_1u)\k(A)}$,
\item\label{it.scond}
$4\sqrt{m}u < 1$ and $16m\sqrt{n}u\scond(\At)< 1${, and}
\item\label{it.jacobi-growth}
{for any matrix $B$ to which the one-sided Jacobi algorithm is applied in this paper, $\max_{0\le k\le M}\scond(B\iter{k}) \le c_J\scond(B)$, where $c_J$ is a modest constant.}
\end{enumerate}
\end{assumption}

\begin{remark}
    Assumption~\ref{it.6nukA} guarantees that $A$ is not so ill-conditioned as to cause the preconditioned matrix $\Atcomp$ to be rank-deficient, as shown in Lemma~\ref{lem.Atcomp-full-rank}. 
    {Assumption~\ref{it.gammah<u} ensures that the error due to the
high-precision matrix--matrix multiplication is sufficiently small
to be absorbed into the working-precision term in our error analysis;
see section~\ref{sec.size-alpha-choice-uh}.} 
    We use Assumption~\ref{it.scond} to relate $\scond(\Atcomp)$ and $\scond(\At)$. 
    {The motivation for Assumption~\ref{it.jacobi-growth} is discussed in Remark~\ref{rem.jacobi-growth}.}
    Note that these assumptions are necessary for our theoretical analysis, but do not appear to be necessary in practice. Indeed, we see in Figure~\ref{subfig.nessie/whiskycorr} that even for a numerically rank-deficient matrix, Algorithm~\ref{alg.mp-precond-onesided-Jacobi} can still achieve high relative forward accuracy. 
\end{remark}

\begin{remark}\label{rem.jacobi-growth}
The bound in~\eqref{eqn:demmelveselicbound} is stated in a simplified form. More precisely, the rounding error analysis in~\cite[Cor.~4.2]{deve92} gives a bound proportional to
\begin{equation}\notag
    u\max_{0\le k\le M}\scond(A\iter{k}),
\end{equation}
where $A\iter{k}$ is the matrix after $k$ Jacobi rotations, $M$ is the number of Jacobi rotations before the stopping criterion~\eqref{eq.one-sided-stopping-critertion} is satisfied, and $A\iter{0}=A$.
Extensive numerical evidence shows that $\max_{0\le k\le M}\scond(A\iter{k})$ is generally not much larger than $\scond(A)$ in practice \cite[sect.~7.4]{deve92} \cite[chap.~5]{slap92}, although a general theoretical explanation is currently not available; see also~\cite{drma20a}.
On the other hand, this maximum can be avoided by applying the one-sided Jacobi algorithm on the left of the input matrix, equivalently on the right of its transpose, in which case the bound depends only on the initial scaled condition number $\scond(A)$~\cite[sect.~3]{math95}~\cite[sect.~4.3]{drma20a}.
\end{remark}

Our analysis focuses on
the relative forward error of the $k$th largest singular value, $\ferrk$,
which is defined as%
\begin{equation}\label{eq.ferrk}
  \ferrk = \frac{\abs{\wh{\sigma}_k(\Atcomp)-\sigma_k(A)}}{\sigma_k(A)}.
\end{equation}
Here, $\Atcomp$ is the computed product $\At=A \Vt$, as discussed in Algorithm \ref{alg.mp-precond-onesided-Jacobi}.

The following result is the main theorem of this paper,
which is related to the main theorem of~\cite{htwz25} on the two-sided variant. The proof is similar at a high level, but the technical details are different.

\begin{theorem}
\label{thm.main}
Let $A \in \R\mn$ $(m\ge n)$ be of full rank,
and let $\At = A\Vt$ be the preconditioned matrix
as in Algorithm~\ref{alg.mp-precond-onesided-Jacobi}.
If Assumption~\ref{ass.main} holds,
then the forward error $\ferrk$ defined in~\eqref{eq.ferrk} satisfies 
\begin{equation}\label{eq.main-result}
  \ferrk \le p_{1} u + p_{2} \scond(\At) u,
\end{equation}
provided that both terms on the right-hand side are less than $1$.
\end{theorem}
The first and second terms of the bound in~\eqref{eq.main-result} describe the error arising from
the orthogonality of the preconditioner $\Vt$
and the Algorithm~\ref{alg.mp-precond-onesided-Jacobi} itself,
respectively.

We decompose $\ferrk$ as follows,
\begin{align}
  \ferrk 
  & \notag \le 
    \underbrace{\frac{\abs{\wh{\sigma}_k(\Atcomp)-\sigma_k(\Atcomp)}}{\sigma_k(A)}}_{E_1}
    + \underbrace{\frac{\abs{\sigma_k(\Atcomp)-\sigma_k(\At)}}{\sigma_k(A)}}_{E_2}
    + \underbrace{\frac{\abs{\sigma_k(\At)-\sigma_k(A)}}{\sigma_k(A)}}_{E_3}. \label{eq.decompose-ferrk}
\end{align}
A bound for $E_3$ is provided by the ``relative'' Weyl theorem 
for multiplicative perturbations~\cite[Cor.~5.2]{demm97-ANLA},
which yields 
\begin{equation}\label{eq.Error-E3}
  E_3 \le \tnorm{\Vt\tp\Vt-I} \le p_1 u. 
\end{equation}

We begin the analysis of $E_1$ and $E_2$ by
bounding $\abs{\Delta \At}$,
the componentwise absolute values of the perturbation
$\Delta \At = \Atcomp - \At$.

\begin{lemma}
\label{lem.size-of-DeltaAt}
Let $\At, \Atcomp \in \R\mn$ be as in
Algorithm~\ref{alg.mp-precond-onesided-Jacobi}.
Then
\begin{equation}\label{eq.error-Atcomp-At-lemma}
  \Atcomp = \At + \Delta \At, \qquad
  \abs{\Delta\At} \le u \abs{\At} + 2\gammah \abs{A}\abs{\Vt}.
\end{equation}
\end{lemma}

\begin{proof}
In line~\ref{li.compute-Atcomp} of
Algorithm~\ref{alg.mp-precond-onesided-Jacobi}, 
$A$ and $\Vt$ are promoted from working precision to high precision,
which is error free~\cite[p.~43]{over25-ncifpa2}.
Then the matrix $\Athcomp$ 
(the computed $\At$ at high precision) satisfies~\cite[Eq.~3.13]{high02-ASNA2},
\begin{equation}\label{eq.matrix-product-high-prec}
  \abs{\Athcomp - \At} \le \gammah \abs{A} \abs{\Vt}.
\end{equation}
Finally we treat the demotion of $\Athcomp$ to working precision $u$  as rounding $\Athcomp$ at working precision, which yields~\cite[Thm.~2.2]{high02-ASNA2}
\begin{equation}\label{eq.round-at-working-prec}
  \abs{\Athcomp - \Atcomp} \le u \abs{\Athcomp}.
\end{equation}
Combining inequalities~\eqref{eq.matrix-product-high-prec}
and~\eqref{eq.round-at-working-prec} yields
\begin{equation}\label{eq.difference-Atcomp-At}
  \abs{\Atcomp - \At}
  \le \abs{\Atcomp-\Athcomp} + \abs{\Athcomp - \At}
  \le  \gamma_h \abs{A} \abs{\Vt} + u\abs{\Athcomp}. 
\end{equation}
Moreover, it follows from~\eqref{eq.matrix-product-high-prec} that $\abs{\Athcomp}=\abs{\Athcomp-\At+\At} \le \gammah\abs{A}\abs{\Vt} + \abs{\At}$.
Substituting this into \eqref{eq.difference-Atcomp-At} yields
\begin{equation}\notag
  \abs{\Atcomp-\At}
  \le \gammah \abs{A} \abs{\Vt} + u (|\At| + \gammah |A||\Vt|)
  < u \abs{\At} + 2\gammah \abs{A}\abs{\Vt},
\end{equation}
where the last inequality follows from $u\gammah < \gammah$. 
\end{proof}


\subsection{Analysis of \texorpdfstring{$E_1$}{E1}}
\label{sec.analysis-E1}
We now bound $E_1$ in~\eqref{eq.decompose-ferrk}.
Let us first rewrite $E_{1}$ as
\begin{equation}\notag
  E_{1} =
  \frac{\abs{\wh{\sigma}_k(\Atcomp)-\sigma_k(\Atcomp)}}{\sigma_k(A)}
  \le
  \frac{|\wh{\sigma}_k(\Atcomp)-\sigma_k(\Atcomp)|}{\sigma_{k}(\Atcomp)}
\cdot
  \frac{\sigma_{k}(\Atcomp)}{\sigma_{k}(A)}
\end{equation}

The following lemmas are important for the analysis that follows.

\begin{lemma}
    \label{lem.conseq-of-VV-I}
    If $\tnorm{\wt{V}\tp \wt{V} - I} \le p_{1} u < 1$, then $\max_j \tnorm{\wt{v}_j} \leq \|\Vt\|_2 \leq 1 + p_1u$, where $\wt{v}_j$ is the $j$th column of $\Vt$.
\end{lemma}
\begin{proof}
By \cite[Eq.~6.12]{high02-ASNA2}, we have
$\max_j\tnorm{\wt{v}_j} \le \tnorm{\Vt}$.
{
For the second inequality, we have
\begin{equation}\notag
  \tnorm{\Vt}^{2}
  = \tnorm{\Vt\tp\Vt}
  \le \tnorm{I} + \tnorm{\Vt\tp\Vt - I}
  \le 1 + p_1 u.
\end{equation}
Therefore $\tnorm{\Vt} \le (1+p_{1}u)^{1/2} \le 1+p_{1}u$.}
\end{proof}

\begin{lemma} \label{lem.Atcomp-full-rank}
Let $\Atcomp \in \R\mn$ be as in Algorithm~\ref{alg.mp-precond-onesided-Jacobi}. If Assumption~\ref{it.6nukA} and $\gamma_h < u$ hold, then $\Atcomp$ has full rank.
\end{lemma}

\begin{proof}
Equation \eqref{eq.error-Atcomp-At-lemma} implies the normwise bound,
\begin{equation}\notag
  \fnorm{\Delta\At}
  \le u \fnorm{\At} + 2\gammah \fnorm{A} \fnorm{\Vt}
  \le n(u + 2\gammah) \tnorm{A} \tnorm{\Vt}.
\end{equation}
Using Lemma~\ref{lem.conseq-of-VV-I} gives 
\begin{equation}\label{eq.normwise-bound-on-DeltaA}
  \fnorm{\Delta \At}
  \le n(u + 2\gammah) (1 + p_1 u) \tnorm{A}
  \le 2nu \tnorm{A} + 4n\gammah \tnorm{A} 
  \le 6nu \tnorm{A},
\end{equation}
where the last inequality is due to $\gammah < u$.
Let us now examine the smallest singular value of $\Atcomp$.
By Weyl's theorem~\cite[Cor.~5.1]{demm97-ANLA}, we have 
\begin{equation}\label{eq.boundAtcomp}
  \sigma_{n}(\Atcomp) = \sigma_{n}(\At + \Delta \At) \ge \sigma_{n}(\At) - \tnorm{\Delta\At}
  \ge \sigma_{n}(\At) - \fnorm{\Delta\At}.
\end{equation}
By the relative Weyl theorem and \eqref{prec-assumpt-orthogonality},
\begin{equation}
    \label{eq.Weyl-on-At-A}
    -p_1u\le-\tnorm{\Vt\tp\Vt-I} \le \frac{\sigma_n(\At)}{\sigma_n(A)}-1\le\tnorm{\Vt\tp\Vt-I} \le p_1u,
\end{equation}
which gives $\sigma_n(\At)\ge (1-p_1u)\sigma_n(A)$.
By substituting~\eqref{eq.normwise-bound-on-DeltaA} into \eqref{eq.boundAtcomp}, we have 
\begin{equation}\label{eq.1111}
  \sigma_{n}(\Atcomp) \ge \sigma_{n}(\At) - 6nu\tnorm{A} \ge
  (1-p_{1}u) \sigma_{n}(A) - 6nu\tnorm{A}.
\end{equation}
Assumption~\ref{it.6nukA} implies $6nu\tnorm{A} < \sigma_{n}(A)(1-p_{1}u)$,
which ensures $\Atcomp$ is full rank, as the right-hand side of~\eqref{eq.1111} becomes greater than zero. 
\end{proof}

Since the matrix $\Atcomp$ has full rank, we can use the rounding error analysis of the one-sided Jacobi algorithm with stopping criterion \eqref{eq.one-sided-stopping-critertion}.
By~\cite[Cor.~4.2]{deve92} {and Assumption~\ref{it.jacobi-growth}}, we have%
\begin{equation}\label{eq.deve92-sigma-bound}
  \frac{|{\wh{\sigma}_k(\Atcomp)-\sigma_k(\Atcomp)}|}{\sigma_k(\Atcomp)}
  \le p_3 \scond(\Atcomp) u,
\end{equation}
provided that $p_{3} \scond(\Atcomp) u < 1$.
Substituting~\eqref{eq.deve92-sigma-bound} into $E_1$ gives 
\begin{equation}\label{eq.E1-into-ratio} 
  E_1 \le p_3 u \scond(\Atcomp) \, \frac{\sigma_k(\Atcomp)}{\sigma_k(A)}. 
\end{equation}
Using the bound \eqref{eq.normwise-bound-on-DeltaA} on $\fnorm{\Delta\At}$
and \eqref{eq.Weyl-on-At-A},  
we are able to control the ratio as follows, 
\begin{equation}
    \notag
    \frac{\sigma_k(\Atcomp)}{\sigma_k(A)} 
    \le \frac{\sigma_k(\At)}{\sigma_k(A)} + \frac{\fnorm{\Delta\At}}{\sigma_k(A)}
    \le 1+p_1 u + \frac{6nu\tnorm{A}}{\sigma_n(A)}
    \le 1 + p_1 u  + 6nu\k(A).
\end{equation}
Substituting this into~\eqref{eq.E1-into-ratio} gives 
$ E_1 \le \big( 1 + p_1u + 6nu\kappa_2(A) \big) \, p_3 u \scond(\Atcomp)$. 
Assumption~\ref{it.6nukA} ensures that $6nu\k(A) < 1$, and since $p_1u < 1$ holds, redefining $p_3$ as $3p_3$ leads to
\begin{equation}\label{eq.Error-E1}
  E_1 \le p_3 u \scond(\Atcomp).
\end{equation}

\subsection{Bound on the forward error \texorpdfstring{$E_2$}{E2}}
\label{sec.bound-on-E2}

Using Lemma~\ref{lem.size-of-DeltaAt}, we rewrite $E_2$ as 
\begin{equation}
    \label{eq.E2-decompose}
    E_2 = \frac{|\sigma_k(\At+\Delta\At)-\sigma_k(\At)|}{\sigma_k(A)}
    \le \frac{|\sigma_k(\At+\Delta\At)-\sigma_k(\At)|}{\sigma_k(\At)}
    \cdot \frac{\sigma_k(\At)}{\sigma_k(A)}.
\end{equation}
The difficulty lies in bounding the first ratio, 
which requires the following theorem.
\begin{theorem}[{\cite[Thm.~2.14]{deve92}}]
\label{thm.deve92-perturb-svals}
Let $\At$ be a general full-rank matrix, and let $\Dt = \diag(\norm{\at_i}\inv)$
so that the columns of $\At \Dt$ have unit $2$-norm.
If $\tnorm{\Delta \At \Dt} < \sigma_n(\At \Dt)$, then %
\begin{equation}\notag
  \frac{\abs{\sigma_k(\At + \Delta \At)-\sigma_k(\At)}}{\sigma_k(\At)}
  \le \scond(\At) \cdot \tnorm{\Delta \At \Dt}.
\end{equation}
\end{theorem}
We see that it suffices to study the behavior of $\tnorm{\Delta\At \Dt}$.
We 
rewrite~\eqref{eq.error-Atcomp-At-lemma} as
$
  \abs{\Delta \at_{ij}} \le u \abs{\at_{ij}} +
  2\gammah \big( \abs{A} \abs{\Vt} \big)_{ij}
$
so that 
\begin{equation}\notag
  \frac{\abs{\Delta\at_{ij}}}{\tnorm{\at_j}}
  \le u \frac{\abs{\at_{ij}}}{\tnorm{\at_j}}
  + 2\gammah \frac{\big( \abs{A} \abs{\Vt} \big)_{ij}}{\tnorm{\at_j}}.
\end{equation}
Since $\abs{\at_{ij}}/\tnorm{\at_j} \le 1$, 
\begin{equation}\label{eq.definition-alpha}
  \frac{\abs{\Delta\at_{ij}}}{\tnorm{\at_j}}
  \le u + 2\gammah\alpha, \quad  \text{ where } \ 
  \alpha = \max_{i,j}
  \frac{\big( \abs{A} \abs{\Vt} \big)_{ij}}{\tnorm{\at_j}}.
\end{equation}
We can interpret~\eqref{eq.definition-alpha} as
$\abs{\Delta \At \Dt} \le (u + 2\gammah\alpha) G$,
where $G \in \R^{m\times n}$ and $g_{ij} = 1$.
By~\cite[Lem.~6.6~(b)]{high02-ASNA2}, we have
\begin{equation}\label{eq.bound-for-DeltaADA}
  \tnorm{\Delta \At \Dt} \le (u + 2\gammah\alpha)\tnorm{G} 
  = \sqrt{mn}(u + 2\gammah\alpha).
\end{equation}
Note that $\tnorm{\At \Dt} \ge \tnorm{\At \Dt e_1} = 1$. Hence,
\begin{equation}
    \notag
    \tnorm{\Delta\At \Dt} 
    \le \sqrt{mn}(u + 2\gammah \alpha) 
    \le \sqrt{mn}(u + 2\gammah \alpha) \tnorm{\At \Dt}.
\end{equation}
For Theorem~\ref{thm.deve92-perturb-svals} to apply, we need $\tnorm{\Delta\At \Dt} < \sigma_n(\At \Dt)$ which is implied by $\sqrt{mn}(u+2\gammah\alpha)\scond(\At) < 1$. Provided this inequality holds,
we can apply Theorem~\ref{thm.deve92-perturb-svals} and the bound \eqref{eq.bound-for-DeltaADA} for $\tnorm{\Delta\At \Dt}$, yielding
\begin{equation}\label{eq.eqn1}
  \frac{\abs{\sigma_k(\Atcomp)-\sigma_k(\At)}}{\sigma_k(\At)}
  \le \sqrt{mn}(u + 2\gammah \alpha)\scond(\At).
\end{equation}
Substituting \eqref{eq.eqn1} into 
the decomposition of $E_2$ \eqref{eq.E2-decompose} gives 
\begin{equation}\label{eq.E2-intermediate}
  E_2 =
  \frac{\abs{\sigma_k(\Atcomp)-\sigma_k(\At)}}{\sigma_k(A)}
  \le \sqrt{mn}(u + 2\gammah \alpha)\scond(\At)
  \cdot \frac{\sigma_k(\At)}{\sigma_k(A)}
\end{equation}
Using \eqref{eq.Weyl-on-At-A} and the assumption $p_1u<1$ gives $\sigma_k(\At)/\sigma_k(A) \le 1+p_1u < 2$,
and \eqref{eq.E2-intermediate} becomes 
\begin{equation}\label{eq.E2-Error}
  E_2 \le \big(2\sqrt{mn}u + 4\sqrt{mn}\gammah \alpha\big) \scond(\At), 
\end{equation}
provided that the right-hand side is less than 1.

\subsection{Intermediate bound}
Summarizing the results for $E_1$~\eqref{eq.Error-E1},
$E_2$~\eqref{eq.E2-Error}, and $E_3$~\eqref{eq.Error-E3},
we have the following result.

\begin{lemma}
\label{lem.intermediate-bound}
Let $\At$, $\Atcomp \in \R\mn$, $\Vt\in \R\nn$ be as in
Algorithm~\ref{alg.mp-precond-onesided-Jacobi}.
Let $\alpha$ be defined as in~\eqref{eq.definition-alpha}.
If Assumption~\ref{it.6nukA} and $\gammah  < u$ hold,
then
\begin{equation}\label{eq.intermediate-bound-e123}
  \ferrk \le p_3 u \scond(\Atcomp) +
  \big(2\sqrt{mn}u + 4\sqrt{mn}\gammah \alpha\big) \scond(\At) 
  + p_1u,
\end{equation}
provided that each term on the right-hand side is less than $1$.
\end{lemma}

For the rest of this section,
we provide a choice of $\uh$ that guarantees
we can drop the dependence on $\alpha$,
and a relationship between $\scond(\Atcomp)$ and $\scond(\At)$
to make the bound concise.

\subsubsection{The size of \texorpdfstring{$\alpha$}{alpha} 
    and choice of \texorpdfstring{$\uh$}{uh}}
\label{sec.size-alpha-choice-uh}
The quantity $\alpha$ in~\eqref{eq.definition-alpha} can be bounded by
$\alpha\le \max_{i,j}\big( \abs{A} \abs{\Vt} \big)_{ij}/
  \min_{j}\tnorm{\at_j}$.
The definition of the max norm~\cite[Tab.~6.2 \& p.~110]{high02-ASNA2} gives
$  \max_{i,j}\big(\abs{A}\abs{\Vt}\big)_{ij}
  = \mnorm{\abs{A}\abs{\Vt}}
  = \norm{|A||\Vt|}_{1,\infty}
$.
Using the property of the mixed subordinate matrix norm~\cite[p.~110]{high02-ASNA2}, 
we have
    $\norm{|A||\Vt|}_{1,\infty} 
    \le \norm{|A|}_{2,\infty} \norm{|\Vt|}_{1,2}$.
It follows from \cite[Prob.~6.11]{high02-ASNA2} that
\begin{align*}
    \notag
    \norm{|A|}_{2,\infty} &= \norm{A}_{2,\infty} \le \tnorm{A},\\
    \norm{|\Vt|}_{1,2} & = \max_{j} \tnorm{(|\Vt|)_j } = 
    \max_{j} \tnorm{|\wt{v}_j|} = \max_{j} \tnorm{\wt{v}_j} \le 1+p_1u,
\end{align*}
where the last inequality is due to Lemma~\ref{lem.conseq-of-VV-I}. By the min-max principle for singular values,
\begin{equation}\label{eqn.sigman(A)}
\sigma_n(\At) = \min_{\|v\|_2=1} \|\At v\|_2 \leq \min_j \|\At e_j\|_2 = \min_j\tnorm{\at_j}.
\end{equation} In conclusion, we have
\begin{equation}\notag
  \alpha 
  \le \frac{(1+p_1u)\tnorm{A}}{\sigma_n(\At)}
  \le \frac{1 + p_1u}{1-p_1u}\k(A). 
\end{equation}

The next corollary incorporates this bound on $\alpha$
with our intermediate bound (Lemma~\ref{lem.intermediate-bound})
to provide conditions on $\uh$ to allow us to drop the dependence on $\alpha$.

\begin{corollary}
\label{cor.intermediate-corollary}
Let $\At,\Atcomp \in \R\mn$, and $\Vt \in \R\nn$ be as in
Algorithm~\ref{alg.mp-precond-onesided-Jacobi}.
If Assumptions~\ref{it.6nukA} and~\ref{it.gammah<u} hold,
then
\begin{equation}\notag
  \ferrk \le
   p_3 u\scond(\Atcomp) + 3\sqrt{mn}u\scond(\At) + p_1u,
\end{equation}
provided that each of the three terms on the right-hand side is less than $1$.
\end{corollary}

\begin{proof}
From~\eqref{eq.intermediate-bound-e123}, 
it is sufficient to prove $4\gammah\alpha < u$.
If we assume that $\gammah$ is as in Assumption~\ref{it.gammah<u},
we have 
\begin{equation}
    \notag
    4\gammah\alpha 
    < 4 \bigg( \frac{(1-p_1u)u}{4 (1+p_1u)\k(A)} \bigg)
    \bigg(\frac{1 + p_1u}{1-p_1u} \k(A)\bigg) = u.
\end{equation}
\end{proof}

\subsubsection{Relationship between $\scond(\At)$ and $\scond(\Atcomp)$}
\label{sec.rel-between-scond-At-Atcomp}

\begin{lemma}
\label{lem.rel-betw-At-Atcomp}
Let $\At,\Atcomp \in \R^{m\times n}$ be the matrices defined in
Algorithm~\ref{alg.mp-precond-onesided-Jacobi},
and let $\Delta\At$ be as defined in \eqref{eq.error-Atcomp-At-lemma}.
If Assumption~\ref{ass.main} holds, then
\begin{equation}\label{eq.relation-between-scond-At-A}
  \scond(\Atcomp) \le 3 \scond(\At).
\end{equation}
\end{lemma}

\begin{proof}
Let us first find a relationship between
$\tnorm{\at_{j}}$ and $\tnorm{(\Atcomp)_{j}}$.
Starting from~\eqref{eq.definition-alpha}
and using $4\gammah\alpha < u$ (shown in the proof of Corollary \ref{cor.intermediate-corollary}), we have
\begin{equation}\label{eq.proof1}
  \abs{\Delta\at_{ij}} \le (u + 2\gammah \alpha) \tnorm{\at_{j}}
  \le 2u \tnorm{\at_{j}}.
\end{equation}
Now we use the same argument as in section~\ref{sec.bound-on-E2}.
Consider a vector $g \in \R^{m}$ with $g_{i} = 1$,
then we have $|\Delta\at_{j}| \le 2u \tnorm{\at_{j}} g$,
which leads to
\begin{equation}\label{eq.delta-atj-2norm}
  \tnorm{\Delta\at_{j}} \le
  \big\|{2u\tnorm{\at_{j}} g}\big\|_{2} = 2\sqrt{m}u\tnorm{\at_{j}}
\end{equation}
By the triangle inequality, we have
$\tnorm{(\Atcomp)_{j}} \le \tnorm{\at_{j}} + \tnorm{\Delta\at_{j}}$, which implies
\begin{equation}\label{eq.relation-j-th-col-Atcomp-At-2}
  \big| \tnorm{(\Atcomp)_{j}} - \tnorm{\at_{j}} \big|
  \le \tnorm{\Delta\at_{j}}. 
\end{equation}
Based on \eqref{eq.delta-atj-2norm}, we can bound $\tnorm{(\Atcomp)_j}$ in terms of $\tnorm{\at_j}$,
\begin{equation}\label{eq.relation-j-th-col-Atcomp-At-1}
  \tnorm{(\Atcomp)_{j}}
  = \tnorm{\at_{j} + \Delta\at_{j}}
  \le (1+2\sqrt{m}u)\tnorm{\at_{j}},
\end{equation}
and similarly 
\begin{equation}
    \tnorm{(\Atcomp)_{j}}
  \ge \tnorm{\at_{j}} -\tnorm{\Delta\at_{j}}
  \ge (1-2\sqrt{m}u)\tnorm{\at_{j}}.
\end{equation}

Now, let us focus on $\scond(\Atcomp)$. Define
\begin{equation}
D_1 = \diag(\tnorm{\at_{i}}^{-1})\text{ and }D_{2} = \diag(\tnorm{(\Atcomp)_{i}}^{-1}),
\end{equation}
and write $\Atcomp D_{2} = (\At + \Delta\At)D_2 = \At D_1+F$, where
\begin{equation}\label{eq.def-of-F}
  {F} = \At(D_{2} - D_{1}) + \Delta\At D_{2}.
\end{equation}
Then by Weyl's theorem,
we obtain an upper bound for $\scond(\Atcomp)$:
\begin{equation}\notag
  \scond(\Atcomp)
  = \frac{\sigma_{1}(\Atcomp D_{2})}{\sigma_{n}(\Atcomp D_{2})}
  = \frac{\sigma_{1}(\At D_{1} + {F})}{
  \sigma_{n}(\At D_{1} + {F})}
  \le \frac{\sigma_{1}(\At D_{1}) + \tnorm{F}}{
  \sigma_{n}(\At D_{1}) - \tnorm{F}}.
\end{equation}
The remaining task is to bound $\tnorm{{F}}$.
By definition of ${F}$, we have
\begin{align*}
  |f_{ij}|
  & \le \bigg| \frac{\at_{ij}}{\tnorm{(\Atcomp)_{j}}}
    - \frac{\at_{ij}}{\tnorm{\at_{j}}} \bigg|
    + \frac{|\Delta\at_{ij}|}{\tnorm{(\Atcomp)_{j}}} \\
  & = 
    \frac{|\at_{ij}|\,\big|\,\tnorm{\at_{j}}-\tnorm{(\Atcomp)_{j}}\,\big|}{
    \tnorm{\at_{j}}\tnorm{(\Atcomp)_{j}}} 
    + \frac{|\Delta\at_{ij}|}{\tnorm{(\Atcomp)_{j}}} \\
  & \le
    \frac{\big|\,\tnorm{\at_{j}}-\tnorm{(\Atcomp)_{j}}\,\big|}{
    \tnorm{(\Atcomp)_{j}}}
    + \frac{\tnorm{\Delta\at_{j}}}{\tnorm{(\Atcomp)_{j}}}.
\end{align*}
From~\eqref{eq.delta-atj-2norm}
and~\eqref{eq.relation-j-th-col-Atcomp-At-2}, we have
$$
  |f_{ij}| \le
  \frac{2\sqrt{m}u \tnorm{\at_{j}}}{\tnorm{(\Atcomp)_{j}}}
  + \frac{2\sqrt{m}u\tnorm{\at_{j}}}{\tnorm{(\Atcomp)_{j}}}
  \le \frac{4\sqrt{m}u}{1-2\sqrt{m}u}.
$$
The first part of Assumption~\ref{it.scond} enables us to bound $|f_{ij}|$ further by $8\sqrt{m}u$. Then, using the matrix $G$ in section~\ref{sec.bound-on-E2} leads to
\begin{equation}\notag
  \tnorm{F} \le 8 \sqrt{mn} \sqrt{m}u = 8 m\sqrt{n}u.
\end{equation}
The upper bound for $\tnorm{F}$ and the lower bound $\sigma_{1}(\At D_{1}) \ge 1$ imply
$$
  \scond(\Atcomp) \le
  \Bigg(\frac{1 + \frac{8m\sqrt{n}u}{\sigma_{1}(\At D_{1})}}{
  1 - \frac{8m\sqrt{n}u}{\sigma_{n}(\At D_{1})}}\Bigg) \scond(\At)
  \le \bigg(\frac{1+8m\sqrt{n}u}{1-8m\sqrt{n}u\scond(\At)}\bigg) 
  \scond(\At).
$$
Assumption~\ref{it.scond} ensures the denominator is always positive. In addition, the second part of that assumption implies $8m\sqrt{n}u\scond(\At) \le 1/2$. As a result, 
\begin{equation}\notag
  \scond(\Atcomp) \le
  \frac{1 + 1/2}{1-1/2} \scond(\At) = 3 \scond(\At).
\end{equation}
\end{proof}

By substituting~\eqref{eq.relation-between-scond-At-A} into
Corollary~\ref{cor.intermediate-corollary},
and defining $p_{2} = 3p_{3} + 3\sqrt{mn}$,
we have proved Theorem~\ref{thm.main}.

\section{Construction of the preconditioner}
\label{sec.constr-of-prec}

Given $A \in \R\mn$, we aim to construct a preconditioner $\Vt\in\R\nn$
by exploiting a low precision $\ul$
that satisfies the following two properties:
\begin{enumerate}[label=\upshape(P\arabic*),labelwidth=!]
\item\label{it.prop-orth}
$\tnorm{\Vt\tp\Vt - I} \le p_{1}u {<1}$,
as required in Algorithm~\ref{alg.mp-precond-onesided-Jacobi}, and
\item\label{it.prop-smaller-off}
the columns of the preconditioned matrix $\At := A\Vt$
should be {nearly orthogonal} such that 
$
    \off(\At\tp\At) \leq p_5 u_\ell \|A^TA\|_F,
$
where $\off(A) := \fnorm{A - \diag(A)}$. 
\end{enumerate}
These two properties are precisely
the counterparts of the requirements imposed on the preconditioner
in the two-sided Jacobi algorithm~\cite[sect.~3]{htwz25}.
Therefore, we can adapt
the two algorithms for constructing the preconditioner%
~\cite[Alg.~2 \& 3]{htwz25} to the one-sided Jacobi algorithm. The resulting algorithms, Algorithms~\ref{alg.orth-method} and~\ref{alg.low-prec-bidiagonal},
both compute a preconditioner that satisfies  properties \ref{it.prop-orth} and \ref{it.prop-smaller-off}.
This can be proved using the arguments from~\cite[sect.~3]{htwz25}
and by treating the right singular vectors of $A$
as the eigenvectors of $A\tp\!A$. 

\begin{remark}\label{rem.offquantity}
    The $\off$ quantity plays an important role in the analysis of the one-sided Jacobi algorithm. It is known that if $\off(A^TA)$ is sufficiently small then the convergence of the algorithm is quadratic~\cite{deri89}. Therefore, an informal measure of the success of the preconditioning step taking $A$ to $\At = A\Vt$ is the magnitude of $\off(\At^T\At)$.
\end{remark}

\begin{algorithm}[!tbhp]
\caption{Constructing a preconditioner by orthogonalizing
  a right singular vector matrix computed at low precision $\ul$.}
\label{alg.orth-method}
\begin{algorithmic}[1]
\Require{A general matrix $A \in\R\mn$, and two precisions
  $0 < u < \ul$.}
\Ensure{A matrix $\Vt \in \R\nn$ satisfying~\ref{it.prop-orth}
  and~\ref{it.prop-smaller-off}.}
\Statex{}
\State{\label{alg-it.orth-svd-ul}Compute a SVD
  $A = U_{\ell}\Sigma_{\ell}V_{\ell}\tp$ at precision $\ul$.}
\State{Orthogonalize $V_{\ell}$ to $\Vt$ at precision $u$.}
\end{algorithmic}
\end{algorithm}

\begin{algorithm}[!tbhp]
\caption{Constructing a preconditioner using a low precision
  bidiagonalization.}
\label{alg.low-prec-bidiagonal}
\begin{algorithmic}[1]
\Require{A general matrix $A \in \R\mn$, and two precisions
  $0 < u < \ul$.}
\Ensure{A matrix $\Vt \in \R\nn$ satisfying~\ref{it.prop-orth}
  and~\ref{it.prop-smaller-off}.}
\Statex{}
\State{Bidiagonalize $A$ to $B_{\ell}$
  using the Householder method at precision $\ul$.
  Store the Householder vectors corresponding to the
  Householder matrices used to zero out the upper part of $A$
  in the columns of $Q \in \R^{n\times (n-2)}$.}
\State{Promote $Q$ to precision $u$.}
\State{Compute a SVD $U_{\ell}\Sigma_{\ell}V_{\ell}\tp$ of $B_{\ell}$
  at precision $u$.}
\State{Form $\Vt$ by applying the Householder matrices,
  constructed using columns of $Q$, to $V_{\ell}$ at precision $u$.}
\end{algorithmic}
\end{algorithm}

\begin{remark}
In each algorithm, the SVD does not need to be computed entirely
as we only need the right singular vectors.
Namely, computing $\Sigma_{\ell}$ and $V_{\ell}$ (without accumulating the left orthogonal transformations into $U_\ell$) is sufficient~\cite[Fig.~8.6.1]{gova13-MC4}.
\end{remark}

\begin{remark}
\label{rmk.R-bidiagonalization}
Sometimes, if $m \gg n$, it is advisable to
first perform a reduced QR factorization $A = QR$,
where $Q \in \R\mn$ has orthonormal columns and
$R \in \R\nn$ is upper triangular,
and then bidiagonalize $R$ to get the SVD~\cite[sect.~5.4.9]{gova13-MC4}.
To the best of our knowledge, such a routine does not exist in LAPACK~\cite{abbb99-LUG3}, but it is available in the nAG Library as \texttt{f02wuf}.\footnote{\url{https://support.nag.com/numeric/nl/nagdoc_latest/flhtml/f02/f02wuf.html}}
Givens rotations are used to reduce $R$ to bidiagonal form~\cite{chan82}, and to obtain a numerically orthogonal matrix $\Vt$ at working precision, one needs to keep track of all the pivots' values, and thereafter
accumulate the Givens rotation matrices at working precision $u$.
\end{remark}

\subsection{Bounding $\scond(\At)$}\label{sec.bound-scond-at}
\def\Dt{\wt{D}}

The $\off$ metric was used in \cite[sec.~2.5]{htwz25} to bound the two-sided scaled condition number of $\At$. Similarly, we can bound the one-sided scaled condition number $\scond(\At)$ if $\off(\At^T\At)$ is sufficiently small.

\begin{proposition}\label{lem.boundscond}
    If $\off(\At^T\At)/\min_j \|\at_j\|^2_2 \leq \theta$ for some $\theta < 1$, then $\scond(\At) \leq (1 + \theta)^{1/2}/(1-\theta)^{1/2}$.
\end{proposition}    
    \begin{proof}
First note that, for $\Dt = \diag(\|\at_j\|_2^{-1})$,
   \begin{equation}\label{eq.htwz-treatment-off}
       \|\Dt \At^T \At \Dt - I \|_2 \leq \|\Dt\|_2^2 \|\At^T\At - \Dt^{-2}\|_F = \off(\At^T\At)/\min_j \|\at_j\|_2^2.
   \end{equation}
Therefore, $\scond(\At)^2=\kappa_2(\Dt \At^T \At \Dt) \leq (1+\theta)/(1-\theta)$,
where we used \cite[Lem.~4.2, Prop.~4.3]{stwu02} for the inequality.
\end{proof}

For two-sided Jacobi, the $\off$ metric leads to useful interpretations of the bound on the two-sided scaled condition number, but in the one-sided case the condition on $\theta$ in Proposition~\ref{lem.boundscond} is too strong --- we have only been able to find a bound on $ \theta$ proportional to $\kappa_2(A)^2$ using techniques from \cite{htwz25}, rather than the ideal linear scaling with $\kappa_2(A)$. 
However, since $\At\Dt$ is close to a matrix with orthonormal columns, we introduce an alternative metric that yields the desired scaling with $\kappa_2(A)$.
\begin{definition}\label{def.obliq}
    Let $A \in \R^{m \times n}$ with $m \geq n$. The \emph{obliquity} of $A$ is $$\obliq(A) := \tnorm{AD - U},$$ where $D = \diag( \tnorm{ a_{j} }\inv )$ and $U$ is an orthogonal polar factor of $AD$ \cite{high08-FM}.
\end{definition}
Note that if $A$ is rank-deficient then the orthogonal polar factor is not unique \cite{high08-FM}, which might lead us to suggest requiring $U$ to be the \emph{canonical} orthogonal polar factor. However, if we write $AD = UH$, where $H = ((AD)^T(AD))^{1/2}$ is the unique symmetric positive semidefinite polar factor of $AD$, then $\obliq(A) = \|UH - U\|_2 = \|H - I\|_2$, which proves that the obliquity is independent of the choice of orthogonal polar factor. 

The obliquity measures how oblique the columns of $A$ are. Conversely, the more orthogonal the columns of $A$ are, the smaller $\obliq(A)$ becomes.
For example, if $A = U_{A}\Sigma_{A}V_{A}\tp$ is an SVD of $A$ then $\obliq(AV_A) = 0$. 

\begin{proposition}\label{lem.deltaAt-cond}
If $\obliq(\At) \leq \theta$ for some $\theta< 1$,
then $\scond(\At) \le (1+\theta)/(1-\theta)$. 
\end{proposition}
\begin{proof}
{Let $U$ be an orthogonal polar factor of $\At\Dt$.}
Since $U$ has orthonormal columns, for any $v \in \R^n$,
$$
\|\At \Dt v\|_2 \leq \|Uv\|_2 + \|\At\Dt v - Uv\|_2  \leq (1+\theta)\|v\|_2,
$$
and
$$
\|\At \Dt v\|_2 \geq \|Uv\|_2 - \|\At\Dt v - Uv\|_2 \geq (1-\theta)\|v\|_2.
$$
This implies that $\sigma_1(\At\Dt) \leq 1+ \theta$ and $\sigma_n(\At\Dt) \geq 1-\theta$, so $\kappa_2(\At\Dt) \leq (1+\theta)/(1-\theta)$.
\end{proof}
By \cite[Lem.~8.17]{high08-FM}, $\obliq(A) \leq \|DA^TAD-I\|_2$, so
\begin{eqnarray*}
    \obliq(A) \leq \|D(A^TA - D^{-2})D\|_F \leq \|D\|_2^2 \|A^TA - D^{-2}\|_F = \frac{\off(A^TA)}{\min_j \|a_j\|_2^2}.
\end{eqnarray*}
Therefore, the conditions of Proposition \ref{lem.deltaAt-cond} are more general than those of Proposition \ref{lem.boundscond}. We will revisit the size of $\scond(\At)$ in section~\ref{sec.prec-and-svd}, where we provide a bound on $\obliq(\At)$ when the preconditioner is constructed using a lower precision $\ul$.

\subsection{Bounding $\obliq(\At)$}
\label{sec.prec-and-svd}
In section~\ref{sec.bound-scond-at}, we proved that
if $\obliq(\At)$ is small, then $\scond(\At)$ can be bounded. In particular, if $\obliq(\At) \leq 1/2$ then $\scond(\At) \leq 3$. In this section, we analyze when we can expect this to occur.

Using similar techniques as in \cite[Lem.~3.2]{htwz25}, one can show that the preconditioners constructed using Algorithms~\ref{alg.orth-method} and~\ref{alg.low-prec-bidiagonal} satisfy
\begin{equation}\label{eq.bwd-err-SVD}
  A + \Delta A = \wt{U}_\ell \Sigma_{\ell} (\Vt + \Delta \Vt)\tp,
\end{equation}
where $\tnorm{\Delta A} \le p_{6} u_{\ell} \tnorm{A}$,
$\tnorm{\Delta\Vt} \le p_{6} u$, and $\wt{U}_\ell = (U_{\ell} + \Delta U_{\ell})$ has exactly orthonormal columns and $(\Vt + \Delta \Vt)$ is exactly orthogonal.
{Here we choose $p_{6}$ such that $p_{6} \ge p_{1}$ to unify the notation.}

\begin{proposition}\label{prop:scondlessthan3}
Suppose that $\ul$ is sufficiently small such that
\begin{equation}\label{eq.additional-assumption}
  5p_{6}(u_{\ell}+u)\kappa_2(A) \leq 1,
\end{equation}
where $p_6$ is as in \eqref{eq.bwd-err-SVD}. Then $\scond(\At) \leq 3$.
\end{proposition}
\begin{proof}
By \cite[Thm~8.4]{high08-FM},
\begin{equation}\label{eqn.obliqandul}
    \obliq(\At) \leq \|\At\Dt - \wt{U}_\ell\|_2,
\end{equation}
so it suffices to bound this norm. By \eqref{eq.bwd-err-SVD},
\begin{eqnarray}
    \notag \At = A\Vt &=& A(\Vt+\Delta\Vt) - A\Delta\Vt  \\
    \notag &=& (A+\Delta A)(\Vt+\Delta\Vt) - \Delta A(\Vt + \Delta\Vt)- A\Delta\Vt\\
    &=&\wt{U}_\ell\Sigma_\ell -\Delta A(\Vt+\Delta\Vt) - A\Delta \Vt. \label{eqn.firstAtpert}
\end{eqnarray}
This implies that
\begin{eqnarray*}
    \At\Dt - \wt{U}_\ell &=&\wt{U}_\ell\left(\Sigma_\ell - \Dt^{-1} \right)\Dt -\Delta A(\Vt+\Delta\Vt)\Dt - A\Delta \Vt\Dt.\label{eqn.At-UtlSigmatl}
\end{eqnarray*}
Therefore,
\begin{equation}\label{eqn.AD-Ulbound}
\|\At\Dt - \wt{U}_\ell\|_2 \leq \|\Sigma_\ell - \Dt^{-1}\|_2 \|\Dt\|_2 + p_6(\ul+u) \|A\|_2 \|\Dt\|_2.
\end{equation}
Let us bound $\|\Sigma_\ell - \Dt^{-1} \|_2$. If we rewrite
$\Sigma_{\ell}(i,i) =
\tnorm{\Sigma_{\ell}e_{i}} = \tnorm{\wt{U}_\ell\Sigma_{\ell}e_{i}}$,
and $\tnorm{\at_{i}} = \tnorm{\At e_{i}}$,
where $e_{i}$ is the $i$th column of the identity matrix,
then by the reverse triangle inequality, we have 
\begin{equation}
  |\Sigma_{\ell}(i,i) - \tnorm{\at_{i}}|
  = \big|\tnorm{\wt{U}_\ell\Sigma_{\ell}e_{i}} - \tnorm{\At e_{i}}\big|
  \le \tnorm{\wt{U}_\ell\Sigma_{\ell}e_{i} - \At e_{i}}
  \le \tnorm{\wt{U}_\ell\Sigma_{\ell} - \At}.\label{eqn.Sigmali-Di}
\end{equation}
By \eqref{eqn.firstAtpert}, $\tnorm{\At - \wt{U}_\ell\Sigma_{\ell}} \leq p_6(\ul+u)\tnorm{A}$. Combining this with \eqref{eqn.obliqandul}, \eqref{eqn.AD-Ulbound} and \eqref{eqn.Sigmali-Di}, we obtain
\begin{equation}
    \obliq(\At) \leq 2 p_6(\ul+u) \|A\|_2 \|\Dt\|_2.
\end{equation}
By \eqref{eqn.sigman(A)}, $\|\Dt\|_2 \leq 1/\sigma_n(\At)$ and by \eqref{eq.Weyl-on-At-A}, $\sigma_n(\At) \geq (1-p_6 u) \sigma_n(A)$, so
\begin{equation}
    \obliq(\At) \leq \frac{2 p_6(\ul+u) \kappa_2(A)}{1-p_6u}.
\end{equation}
By \eqref{eq.additional-assumption},
$\obliq(\At) \leq \frac{2/5}{1-1/5} = \frac{1}{2}$. The result now follows from Proposition~\ref{lem.deltaAt-cond}.
\end{proof}

\section{Position of the QR factorization}\label{sec.qr}
A QR factorization is sometimes used before applying an SVD algorithm when $m \gg n$ for dimension reduction~\cite[sect.~5.4.9]{gova13-MC4}. In this section we discuss {whether the QR factorization should be performed before or after the preconditioning step in Algorithm~\ref{alg.mp-precond-onesided-Jacobi}.}

\subsection{For the one-sided Jacobi algorithm}
The one-sided Jacobi algorithm can be accelerated on trapezoidal matrices by exploiting their structure~\cite[sect.~2]{drve08ii}.
In addition, the computed upper trapezoidal QR factor $\wh{R}$ of $A$ satisfies
\begin{equation}\label{eq.QR-factorization}
  A + \Delta A = Q \wh{R}, \qquad \tnorm{\Delta a_j} \le \gamma_{cmn} \tnorm{a_j},
  \quad \gamma_{cmn} = \frac{cmnu}{1-cmnu} \in (0,1),
\end{equation}
where $Q \in \R^{m \times m}$ is orthogonal and $c$ is a modest constant~\cite[Thm.~19.4]{high02-ASNA2}. This backward error result allows us to prove the following forward error bound on the singular values computed by the one-sided Jacobi algorithm applied to $\wh{R}$ (instead of applied directly to $A$).

\begin{proposition}\label{lem.Jacobi-on-QR}
Let $\wh{R}$ be the computed upper trapezoidal QR factor of $A$.
Suppose that the one-sided Jacobi algorithm applied to $\wh{R}$
converges after $M$ rotations and is stopped using
criterion~\eqref{eq.one-sided-stopping-critertion} with
$\tol=O(u)$. Let $q(M,n)$ be the factor in the corresponding
relative forward error bound.
{Suppose further that Assumption~\ref{it.jacobi-growth} holds.}
If
\begin{equation}\notag
   \sqrt{n}\gamma_{cmn}\scond(A)<\frac{1}{2}, \qquad
   3q(M,n)u\scond(A)<1,
\end{equation}
then the computed singular values satisfy
\begin{equation}\notag
    \frac{|\sigma_k(A)-\wh{\sigma}_k(\wh{R})|}{\sigma_k(A)}
    \le p_4u\scond(A).
\end{equation}
\end{proposition}

\begin{proof}
Using the triangle inequality, we have, for all $1\le k \le n$,
\begin{equation}
    \label{eq.QR-factorization-triangular-ineqn}
    \frac{|\sigma_k(A) - \wh{\sigma}_k(\wh{R})|}{\sigma_k(A)} \le 
    \frac{|\sigma_k(A)-\sigma_k(\wh{R})|}{\sigma_k(A)} + 
    \frac{|\sigma_k(\wh{R}) - \wh{\sigma}_k(\wh{R})|}{\sigma_k(\wh{R})} \frac{\sigma_k(\wh{R})}{\sigma_k(A)}.
\end{equation}
By the additive perturbation result for singular values~\cite[Cor.~3.7]{ipse98}, we have
\begin{equation}
    \notag 
    \frac{|\sigma_k(A) - \sigma_k(\wh{R})|}{\sigma_k(A)} \le \tnorm{\Delta A D (AD)\pinv}, 
    \qquad 
    D = \diag(\tnorm{a_j}\inv).
\end{equation}
Using~\cite[Eq.~19.12]{high02-ASNA2}, we can rewrite \eqref{eq.QR-factorization} as
\begin{equation}\notag
  (AD + \Delta A D) D\inv = Q \wh{R}, \qquad \tnorm{\Delta AD} \le \wt\gamma_{cmn} \tnorm{AD},
\end{equation}
where $\wt\gamma_{cmn} = \sqrt{n}\gamma_{cmn}$.
As a result, we obtain
\begin{equation}
  \label{eq.QR-fact-svals-sensitivity}
  \frac{|\sigma_k(A) - \sigma_k(\wh{R})|}{\sigma_k(A)} \le \tnorm{\Delta A D} \tnorm{(AD)\pinv}
  \le 
  \wt\gamma_{cmn} \scond(A).
\end{equation}

The final term in \eqref{eq.QR-factorization-triangular-ineqn} consists of $|\sigma_k(\wh{R})-\wh{\sigma}_k(\wh{R})|/\sigma_k(\wh{R})$, which{, by~\cite[Cor.~4.2]{deve92} and Assumption~\ref{it.jacobi-growth},} is bounded by $q(M,n)u \scond(\wh{R})$ with $q(M,n)$ being a polynomial in the number of applied rotations $M$ and the number of columns $n$, and the ratio $\sigma_k(\wh{R})/\sigma_k(A)$, which can be bounded by $1+ \wt\gamma_{cmn}\scond(A)$ using \eqref{eq.QR-fact-svals-sensitivity}. Using a similar argument in the proof of Lemma~\ref{lem.rel-betw-At-Atcomp}, and assuming $\wt\gamma_{cmn}\scond(A) < 1/2$, we have
\begin{equation}
    \notag 
    \scond(\wh{R}) \le \left( \frac{1+\wt\gamma_{cmn} \scond(A)}{1-\wt\gamma_{cmn}\scond(A)}\right) \scond(A) \le 3\scond(A),
\end{equation}
which enables us to further bound the latter term of \eqref{eq.QR-factorization-triangular-ineqn}. In conclusion, we have 
\begin{equation}
    \notag
    \frac{|\sigma_k(A) - \wh{\sigma}_k(\wh{R})|}{\sigma_k(A)} 
    \le p_4u\scond(A), \quad 
    p_4 = \frac{cmn^{3/2}}{1-cmnu}+\frac{9q(M,n)}{2}.
\end{equation}
\end{proof}

Applying a QR factorization before the one-sided Jacobi algorithm therefore does not significantly increase the relative forward error of the singular values. However, it is not as simple for the mixed-precision preconditioned one-sided Jacobi algorithm considered in this paper (Algorithm~\ref{alg.mp-precond-onesided-Jacobi}), as we now discuss.

\subsection{QR before preconditioning}

The proof of Proposition \ref{lem.Jacobi-on-QR}, in particular \eqref{eq.QR-fact-svals-sensitivity}, suggests that if we perform a QR factorization on $A$ before applying the preconditioner $\Vt$, then the relative forward error in the singular values could scale like $\scond(A)$ instead of $\scond(\At)$. We have confirmed numerically that this can happen in practice. This is problematic because $\scond(A)$ can be significantly larger than $\scond(\At)$ (similarly to the examples provided in \cite[Tab.~1]{htwz25}). A remedy for the resulting inaccuracy is to compute the QR factorization in a higher precision, so that the difference between $\scond(A)$ and $\scond(\At)$ can be offset by a smaller unit roundoff, but this could be prohibitively expensive.

\subsection{QR after preconditioning} 
We suggest performing a QR factorization \emph{after} applying the preconditioner $\Vt$ in high precision, before applying the one-sided Jacobi algorithm.
A natural consideration is that if we use the preconditioner as discussed in section \ref{sec.constr-of-prec}, which achieves%
\begin{equation}\label{eq.htwz-preconditioner-property}
  \off \big( \At \tp \At ) \leq p_5 \ul \fnorm{A\tp A},
\end{equation}
where $\ul$ is the precision used to compute the preconditioner, then does this property remain after a QR factorization? The importance of this is discussed in Remark \ref{rem.offquantity}.
\begin{lemma}\label{lem.QRafterprecond}
Let $\At + \Delta \At = Q_{\At}\wh{R}_{\At}$ be a computed QR factorization of $\At$ at working precision $u$ satisfying \eqref{eq.QR-factorization}.
Then,
\begin{equation}\notag
  \off\big(\wh{R}_{\At}\tp \wh{R}_{\At} \big)
  \le \off\big(\At\tp \At\big) + 3{n}^{3/2} \gamma_{cmn} \fnorm{\At\tp\At}.
\end{equation}
\end{lemma}

\begin{proof}
Since $Q_{\At}$ is exactly orthogonal, we have
\begin{equation}\notag
  \off\big(\wh{R}_{\At}\tp \wh{R}_{\At} \big) =
  \off\big( (Q_{\At} \wh{R}_{\At})\tp (Q_{\At} \wh{R}_{\At}) \big) =
  \off\big( (\At + \Delta\At)\tp (\At + \Delta\At)\big).
\end{equation}
After expanding the Gram matrix within the off operator, we need to analyze two matrices, $(\Delta \At) \tp \At$ and $(\Delta \At)\tp (\Delta \At)$. We first have
\begin{equation}\notag
  \big|\big( (\Delta \At)\tp \At \big)_{ij}\big|
  = | (\Delta \at_{i})\tp \at_{j} |
  \le \tnorm{\Delta \at_{i}} \tnorm{\at_{j}}
  \le \gamma_{cmn} \tnorm{\at_{i}} \tnorm{\at_{j}},
\end{equation}
and similarly for $(\Delta \At)\tp (\Delta \At)$,
\begin{equation}\notag
  \big|\big( (\Delta \At)\tp (\Delta \At) \big)_{ij}\big|
  \le \gamma_{cmn}^{2}\tnorm{\at_{i}} \tnorm{\at_{j}}. 
\end{equation}
Using the assumption that $\gamma_{cmn} \le 1$, we finally obtain a componentwise relation
\begin{equation}\notag
  \big| \big( (\At + \Delta\At)\tp (\At + \Delta\At) \big)_{ij} \big|
  \le \big | \big( \At \tp \At \big)_{ij} \big| + 3\gamma_{cmn} \tnorm{\at_{i}} \tnorm{\at_{j}}. 
\end{equation}
Summing over all $i \neq j$, we have
\begin{equation}\label{eq.off-RtR}
  \off(\wh{R}_{\At}\tp \wh{R}_{\At}) \le
  \off ( \At\tp \At )
  + 3\gamma_{cmn} \sum_{i\neq j}^{} \tnorm{\at_{i}} \tnorm{\at_{j}}. 
\end{equation}
We can bound the final summation using the Cauchy--Schwarz inequality as follows.
\begin{equation*}
\sum_{i\neq j}^{} \tnorm{\at_{i}} \tnorm{\at_{j}} \leq \left(\sum_{i=1}^{n} \tnorm{\at_{i}} \right)^2 \leq n \|\At\|_F^2.
\end{equation*}
By another application of the Cauchy--Schwarz inequality, we have
\begin{eqnarray*}
    \fnorm{\At}^{2} &=& \sum_{i=1}^n \sigma_i(\At^T\At) \leq n^{1/2}\left(\sum_{i=1}^n \sigma_i(\At^T\At)^2 \right)^{1/2} = n^{1/2} \|\At^T\At\|_F.
\end{eqnarray*}
Therefore $\sum_{i\neq j}^{} \tnorm{\at_{i}} \tnorm{\at_{j}} \le {n}^{3/2} \fnorm{\At\tp\At}$,
and substituting this back into \eqref{eq.off-RtR} completes the proof.
\end{proof}
\begin{remark}\label{rem.QR}
    Lemma \ref{lem.QRafterprecond} tells us that performing a QR factorization after preconditioning maintains a small $\off$ quantity, and Proposition \ref{lem.Jacobi-on-QR} (with $A$ replaced by $\Atcomp$) guarantees the accuracy of the resulting computed singular values.
\end{remark}

\section{Numerical Experiments} \label{sec.numer-experi}

{Algorithm~\ref{alg.mp3jacobi-qr} provides the practical implementation of Algorithm~\ref{alg.mp-precond-onesided-Jacobi} with a QR factorization.}
\begin{algorithm}[!tbhp]
\caption{Mixed-precision preconditioned one-sided Jacobi algorithm with QR factorization}
\label{alg.mp3jacobi-qr}
\begin{algorithmic}[1]
\Require{A full-rank matrix $A \in \R\mn$ with $m \ge n$, and
  three precisions $\ul$, $u$ and $\uh$ with $0 < \uh < u < \ul$.}
\Ensure{A computed SVD $U\Sigma V\tp$ of $A$.}
\Statex{}
\State{Compute a right singular vector matrix $V_{\ell}$ of $A$
  at precision $\ul$.}
\State{Compute a Householder QR factorization of $V_{\ell}$
  at precision $u$, and take its computed orthogonal factor as $\Vt$.}
\State{%
  Compute the preconditioned matrix $\At$
  by computing the product $A\Vt$ at precision $\uh$,
  which gives $\Athcomp$,
  and then demoting it to precision $u$, which yields $\Atcomp$.}
\If{$m \ge 11n/6$}
\State{Compute a reduced QR factorization
  $\Atcomp = Q_{\At} \wh{R}_{\At}$ at precision $u$.}
\State{%
  Compute an SVD, $\wh{R}_{\At} = U_{J}\Sigma V_J\tp$
  using \inline|DGESVJ| with \inline|JOBA='U'| at precision $u$.}
\State{Construct the left singular vector matrix
  $U = Q_{\At} U_{J}$ at precision $u$.}
\Else{}
\State{%
  Compute an SVD, $\Atcomp = U\Sigma V_J\tp$
  using \inline|DGESVJ| with \inline|JOBA='G'| at precision $u$.}
\EndIf
\State{
  Construct the right singular vector matrix
  $V = \Vt V_J$ at precision $u$.}
\end{algorithmic}
\end{algorithm}

{The first two steps of Algorithm~\ref{alg.mp3jacobi-qr}
implement Algorithm~\ref{alg.orth-method},
with a Householder QR factorization
used for the orthogonalization.
The ratio $11/6$ in the QR crossover criterion
follows the choice made in
the real LAPACK divide-and-conquer SVD drivers
\inline{SGESDD} and \inline{DGESDD}.
We adopt this ratio as a practical heuristic,
and this is not required by our analysis.
Among the three job parameters of
\inline{DGESVJ}, only \inline{JOBA} differs between two scenarios:
\inline|JOBU='U'| and \inline|JOBV='V'| are used in both,
whereas \inline|JOBA='U'| specifies
the upper triangular input $\wh{R}_{\At}$
and \inline|JOBA='G'| specifies the general input $\Atcomp$.}

We present numerical experiments to assess both the {relative forward accuracy} and the speed of Algorithm~\ref{alg.mp3jacobi-qr} using single, double, and quadruple precisions. All experiments were performed in {MATLAB R2025b Update 1 on a MacBook Pro with an M3 Pro chip and 36 GB of RAM}. MATLAB code to produce the results of this section is available at \url{https://github.com/zhengbo0503/Code_twz26a}.

To ensure reproducibility, all test scripts' random number generator were seeded with \verb|rng(1)|.
Quadruple precision was simulated using the Advanpix Multiprecision Computing Toolbox~\cite{mct2023} with \verb|mp.Digits(34)|. We tested the following four algorithms:
\begin{itemize}
    \item MP3JacobiSVD: Algorithm~\ref{alg.mp3jacobi-qr} with
$(\ul,u,\uh)=(\text{single},\text{double},\text{quadruple})$.
    \item \inline{DGESVJ}: LAPACK subroutine for one-sided Jacobi algorithm,
    \item \inline{DGEJSV}: LAPACK subroutine for preconditioned one-sided Jacobi algorithm, and 
    \item MATLAB \texttt{svd}: MATLAB built-in function for computing SVD. 
\end{itemize}
Both \inline{DGESVJ} and \inline{DGEJSV} were invoked in MATLAB through the MEX facility~\cite[sect.~23.7]{hihi16-MG3} using the OpenBLAS LAPACK implementation (Version 0.3.29).
All the reference singular values were computed using MATLAB \texttt{svd} at octuple precision (binary256~\cite{ieee19}), which is simulated by the Advanpix Toolbox with \verb|mp.Digits(71)|.

The command \inline{gallery('randsvd', [m,n], kappa, MODE)} was used to generate random matrices with various sizes, condition numbers \inline{kappa}, and the singular value distributions $\texttt{MODE}\in\{1,2,3,4,5\}$.

{
Our numerical experiments are deliberately not
restricted to matrices satisfying Assumption~\ref{ass.main}.
In particular, they include highly ill-conditioned matrices
and the numerically rank-deficient \inline{whiskycorr} matrix
considered in section~\ref{sec.numerical-experiment.special-test}.
The results show that, on these test matrices,
MP3JacobiSVD continues to achieve high relative accuracy, 
even outside the scope of our analysis.
}

\subsection{Relative forward accuracy} \label{sec.fwdrelacc}
The experiments of this subsection were performed to {assess the relative forward accuracy of the algorithms}. The working precision was double precision with $u\approx 10^{-16}$. 
Since the bound is usually dominated by the scaled condition number term, we neglect the contribution of $p_1 u$. 
To measure the relative forward error for different matrix types, we compute the \textit{maximum relative forward error} for each test matrix, defined as $\max_{1\le k\le n} \ferrk$.
For all experiments in this section, we choose $p_2 = (mn)^{1/2}$. In practice, the exact matrix $\At$ is not available. Instead, we approximate $\scond(\At)$ using $\scond(\Athcomp)$, {where $\Athcomp$ is the product $A\Vt$ computed in quadruple precision. As a result, we use $(mn)^{1/2}u\scond(\Athcomp)$ as a practical reference line.}

\subsubsection{Varying matrix condition number} \label{sec.numer-experi.vary-cond}

\input{figs/varying-kappa}

We generated $A \in \R^{1000\times 800}$, with $\kappa_2(A)$ taking twenty logarithmically spaced values in the range $10^3$ to $10^{15}$ and five different singular value distributions determined by \inline{randsvd}'s MODE. In Figure~\ref{fig.fwderr_diff_cond} we observe that our algorithm was almost always more accurate than the other methods considered, especially for ill-conditioned matrices. When $\kappa_2(A) = 10^{14}$, MP3JacobiSVD can still achieve about $8$ digits of relative accuracy, while the others provide only around $2$ digits.

\subsubsection{Varying matrix size} \label{sec.numer-experi.vary-mat-size}

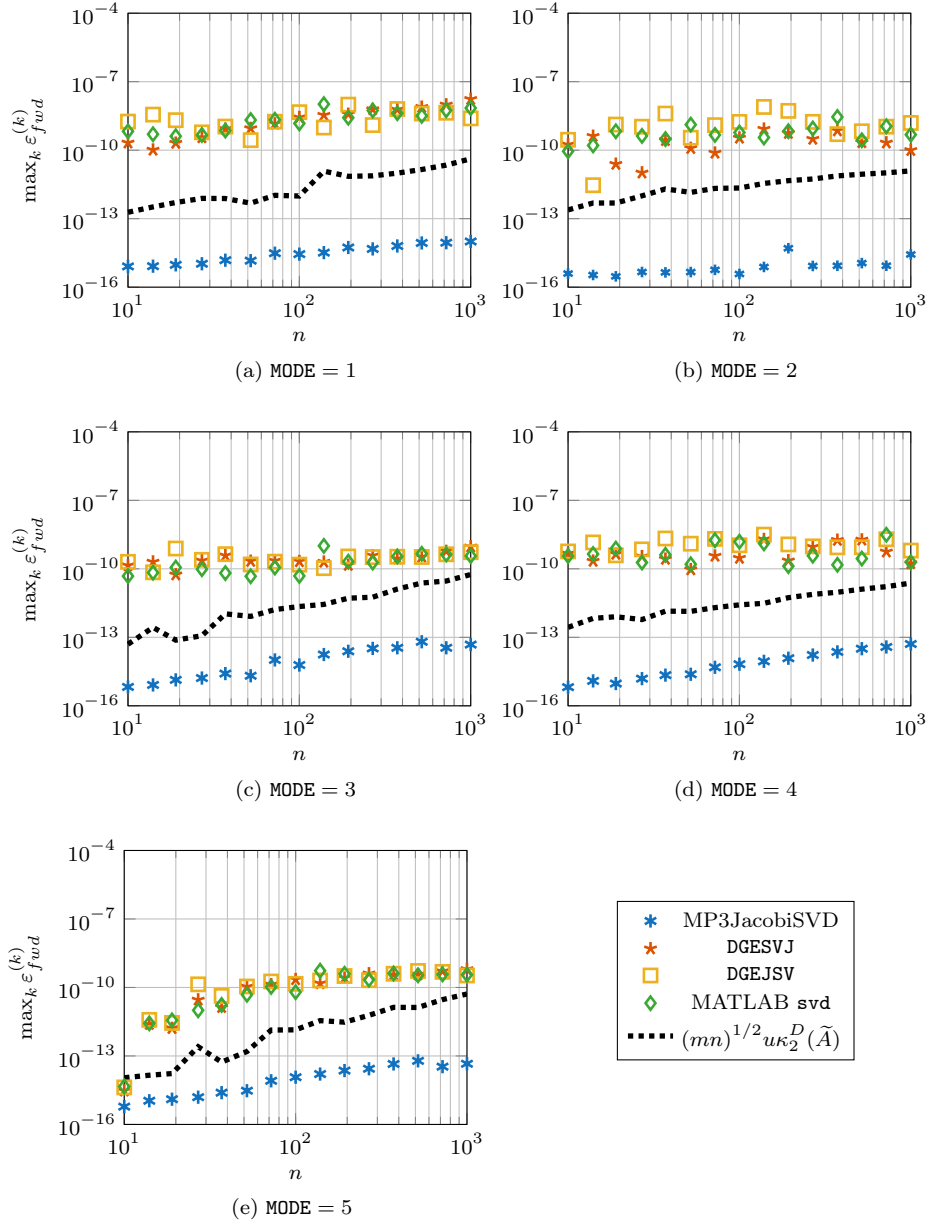
\begin{figure}[!tbhp]
\centering
\footnotesize
\subfloat[$\texttt{MODE} = 1$]{
  \begin{tikzpicture}[trim axis right, trim axis left]
    \begin{axis}[%
xmode=log,
xmin=10,
xmax=1000,
xminorticks=true,
xlabel={$n$},
ymode=log,
ymin=1e-16,
ymax=0.0001,
ytick={ 1e-16,  1e-13,  1e-10,  1e-07, 0.0001},
yminorticks=true,
ylabel={$\max_{k}\ferrk$}
]
\addplot [color=mycolor1, line width=1.0pt, mark size=2.5pt, only marks, mark=asterisk, mark options={solid, mycolor1}, forget plot]
  table[row sep=crcr]{%
10	8.10421790900003e-16\\
14	8.29245242510352e-16\\
19	9.49713116118224e-16\\
27	1.06170166369885e-15\\
37	1.52045057826246e-15\\
52	1.44881344682978e-15\\
72	3.07385770591868e-15\\
100	2.80229209991689e-15\\
139	3.25336793392558e-15\\
193	5.45651448038005e-15\\
268	4.73206049142187e-15\\
373	6.37343841647569e-15\\
518	8.64564060189757e-15\\
720	8.85917856968486e-15\\
1000	1.00388987906343e-14\\
};
\addplot [color=mycolor2, line width=1.0pt, mark size=2.5pt,  only marks, mark=star, mark options={solid, mycolor2}, forget plot]
  table[row sep=crcr]{%
10	2.1087802184469e-10\\
14	1.0276544310303e-10\\
19	2.04989212930905e-10\\
27	3.59945158573457e-10\\
37	8.32488266051102e-10\\
52	9.00888234390018e-10\\
72	1.43096456196056e-09\\
100	2.71519200273342e-09\\
139	3.44374619872361e-09\\
193	3.78440573872287e-09\\
268	6.10839382605527e-09\\
373	5.90773645864234e-09\\
518	7.70334248006251e-09\\
720	9.76822887831941e-09\\
1000	1.7126217155791e-08\\
};
\addplot [color=mycolor3, line width=1.0pt, mark size=2.5pt,  only marks, mark=square, mark options={solid, mycolor3}, forget plot]
  table[row sep=crcr]{%
10	1.7899046689474e-09\\
14	3.62286315215136e-09\\
19	2.05703148299923e-09\\
27	6.03126131812621e-10\\
37	1.08175854468338e-09\\
52	2.75301703666815e-10\\
72	1.76705940864312e-09\\
100	4.61850421489377e-09\\
139	9.80543529094425e-10\\
193	9.91124282744222e-09\\
268	1.25667278697534e-09\\
373	6.35395593413652e-09\\
518	4.00112704150278e-09\\
720	4.35630645233622e-09\\
1000	2.46979775485533e-09\\
};
\addplot [color=mycolor4, line width=1.0pt, mark size=2.5pt,  only marks, mark=diamond, mark options={solid, mycolor4}, forget plot]
  table[row sep=crcr]{%
10	6.70538392464903e-10\\
14	5.004508631309e-10\\
19	4.08832086210082e-10\\
27	4.63159807902025e-10\\
37	7.27316231710983e-10\\
52	2.09138973428016e-09\\
72	2.15041022505941e-09\\
100	1.49159903155702e-09\\
139	1.0397449898252e-08\\
193	2.65049130744879e-09\\
268	5.5122175347056e-09\\
373	4.17261312119662e-09\\
518	3.19408900825703e-09\\
720	5.6123479069002e-09\\
1000	7.01725492006046e-09\\
};
\addplot [color=black, dotted, line width=2.0pt,  forget plot]
  table[row sep=crcr]{%
10	1.90776266673465e-13\\
14	3.276567579419e-13\\
19	5.10312857414551e-13\\
27	7.74783998216569e-13\\
37	7.63531607899776e-13\\
52	4.79596387883633e-13\\
72	1.05236970693594e-12\\
100	9.83424810833748e-13\\
139	1.20742454325897e-11\\
193	7.15682473079675e-12\\
268	7.52674310011829e-12\\
373	9.86516343347916e-12\\
518	1.42692546645628e-11\\
720	2.22117560160248e-11\\
1000	4.19822126400167e-11\\
};
\end{axis}
  \end{tikzpicture}}
\qquad\qquad
\subfloat[$\texttt{MODE} = 2$]{
  \begin{tikzpicture}[trim axis right, trim axis left]
    \begin{axis}[%
xmode=log,
xmin=10,
xmax=1000,
xminorticks=true,
xlabel={$n$},
ymode=log,
ymin=1e-16,
ymax=0.0001,
ytick={ 1e-16,  1e-13,  1e-10,  1e-07, 0.0001},
yminorticks=true
]
\addplot [color=mycolor1, line width=1.0pt, mark size=2.5pt, only marks, mark=asterisk, mark options={solid, mycolor1}, forget plot]
  table[row sep=crcr]{%
10	3.94734435385903e-16\\
14	3.41178200642396e-16\\
19	2.975880360118e-16\\
27	4.61297452210624e-16\\
37	4.43670102505758e-16\\
52	4.58916225743097e-16\\
72	5.73046543490367e-16\\
100	3.78436226570507e-16\\
139	7.68237662298773e-16\\
193	5.0827169961363e-15\\
268	8.4913936572839e-16\\
373	8.65460327920015e-16\\
518	1.13295103667284e-15\\
720	8.69419304181823e-16\\
1000	2.70222102019757e-15\\
};
\addplot [color=mycolor2, line width=1.0pt, mark size=2.5pt,  only marks, mark=star, mark options={solid, mycolor2}, forget plot]
  table[row sep=crcr]{%
10	1.75171809063726e-10\\
14	4.12204788745838e-10\\
19	2.45776193820203e-11\\
27	1.02621184857719e-11\\
37	2.70162627135882e-10\\
52	1.21374324664856e-10\\
72	7.57025934545771e-11\\
100	3.45330431966288e-10\\
139	8.44632858977999e-10\\
193	5.54872304176122e-10\\
268	3.08242780328262e-10\\
373	7.19862618638252e-10\\
518	2.19482652161001e-10\\
720	2.11615589239271e-10\\
1000	9.92531271955913e-11\\
};
\addplot [color=mycolor3, line width=1.0pt, mark size=2.5pt,  only marks, mark=square, mark options={solid, mycolor3}, forget plot]
  table[row sep=crcr]{%
10	2.85551617209681e-10\\
14	2.9212813550469e-12\\
19	1.32595752337795e-09\\
27	1.06752103386482e-09\\
37	4.02312556623731e-09\\
52	3.51171385608271e-10\\
72	1.24918534588035e-09\\
100	1.73243158545368e-09\\
139	7.8332183626273e-09\\
193	5.27197454029507e-09\\
268	1.74255924292421e-09\\
373	5.04161377812237e-10\\
518	6.73959999694595e-10\\
720	1.07112349967005e-09\\
1000	1.55695052382211e-09\\
};
\addplot [color=mycolor4, line width=1.0pt, mark size=2.5pt,  only marks, mark=diamond, mark options={solid, mycolor4}, forget plot]
  table[row sep=crcr]{%
10	8.97311278081926e-11\\
14	1.6173549207025e-10\\
19	6.70465885994911e-10\\
27	4.22448467927341e-10\\
37	3.06544677430963e-10\\
52	1.31315091534434e-09\\
72	4.65777941236119e-10\\
100	5.95880957797649e-10\\
139	3.59652446586584e-10\\
193	6.78391122927401e-10\\
268	8.90738240891813e-10\\
373	2.84487189150343e-09\\
518	2.71720927628079e-10\\
720	1.11384580502747e-09\\
1000	4.78695657237835e-10\\
};
\addplot [color=black, dotted, line width=2.0pt, forget plot]
  table[row sep=crcr]{%
10	2.42257211168608e-13\\
14	4.84861872959255e-13\\
19	4.84738119489742e-13\\
27	1.0010583401435e-12\\
37	2.00078434179987e-12\\
52	1.41813036649705e-12\\
72	2.12142302828513e-12\\
100	2.19997769166893e-12\\
139	3.49776817314093e-12\\
193	4.62556217340291e-12\\
268	5.46416996751799e-12\\
373	7.48250935353498e-12\\
518	8.86959516525968e-12\\
720	1.01655177897946e-11\\
1000	1.24097793253015e-11\\
};
\end{axis}
  \end{tikzpicture}}

\subfloat[$\texttt{MODE} = 3$]{
  \begin{tikzpicture}[trim axis right, trim axis left]
    \begin{axis}[%
xmode=log,
xmin=10,
xmax=1000,
xminorticks=true,
xlabel={$n$},
ymode=log,
ymin=1e-16,
ymax=0.0001,
ytick={ 1e-16,  1e-13,  1e-10,  1e-07, 0.0001},
yminorticks=true,
ylabel={$\max_{k}\ferrk$},
]
\addplot [color=mycolor1, line width=1.0pt, mark size=2.5pt,  only marks, mark=asterisk, mark options={solid, mycolor1}, forget plot]
  table[row sep=crcr]{%
10	6.70573890332404e-16\\
14	8.19993750703761e-16\\
19	1.36704675598269e-15\\
27	1.64599170378877e-15\\
37	2.57956325943601e-15\\
52	2.07141336959538e-15\\
72	1.02214808034635e-14\\
100	6.11820416497971e-15\\
139	1.76311372857429e-14\\
193	2.44610176028744e-14\\
268	3.26984492596702e-14\\
373	3.47686725110715e-14\\
518	6.39315739921635e-14\\
720	3.47798410455825e-14\\
1000	4.80454068393163e-14\\
};
\addplot [color=mycolor2, line width=1.0pt, mark size=2.5pt,  only marks, mark=star, mark options={solid, mycolor2}, forget plot]
  table[row sep=crcr]{%
10	1.3981832594683e-10\\
14	2.01521901775546e-10\\
19	5.56416717716705e-11\\
27	2.28348013299828e-10\\
37	3.81959480067159e-10\\
52	2.16722442929245e-10\\
72	2.25590782387908e-10\\
100	2.10206029064428e-10\\
139	1.99234928492764e-10\\
193	1.45484797198797e-10\\
268	3.80131612049429e-10\\
373	2.78595370163358e-10\\
518	3.35529349451633e-10\\
720	5.85875262515911e-10\\
1000	9.32489637460913e-10\\
};
\addplot [color=mycolor3, line width=1.0pt, mark size=2.5pt,  only marks, mark=square, mark options={solid, mycolor3}, forget plot]
  table[row sep=crcr]{%
10	2.02227283370156e-10\\
14	7.02140932672375e-11\\
19	7.7578039167767e-10\\
27	2.48560162560446e-10\\
37	4.38680743594455e-10\\
52	1.55564598899912e-10\\
72	2.06384385172166e-10\\
100	1.52423489706448e-10\\
139	1.07060462091083e-10\\
193	3.48701798001677e-10\\
268	3.2692353147982e-10\\
373	3.3894118347651e-10\\
518	3.30933032123764e-10\\
720	4.33856695764671e-10\\
1000	5.35539820521732e-10\\
};
\addplot [color=mycolor4, line width=1.0pt, mark size=2.5pt,  only marks, mark=diamond, mark options={solid, mycolor4}, forget plot]
  table[row sep=crcr]{%
10	4.81056681415356e-11\\
14	6.80008888199741e-11\\
19	1.13621076982454e-10\\
27	9.60720726399875e-11\\
37	6.50381323881963e-11\\
52	4.73321656995355e-11\\
72	1.14110705186488e-10\\
100	4.8889266003849e-11\\
139	1.01330590899195e-09\\
193	2.05792651000116e-10\\
268	1.87624389747198e-10\\
373	3.46876142593626e-10\\
518	4.56182875511113e-10\\
720	4.25440521286717e-10\\
1000	3.65359671541873e-10\\
};
\addplot [color=black, dotted, line width=2.0pt, forget plot]
  table[row sep=crcr]{%
10	4.97847124252892e-14\\
14	2.72192487609967e-13\\
19	7.51513390742356e-14\\
27	1.17684652933575e-13\\
37	1.06293653057499e-12\\
52	8.26284427553637e-13\\
72	1.64267399094145e-12\\
100	2.23656272474097e-12\\
139	2.76615355503289e-12\\
193	5.16671531306001e-12\\
268	5.72961091006867e-12\\
373	1.29973870946972e-11\\
518	2.38511960169565e-11\\
720	2.92777954085226e-11\\
1000	5.64607339586827e-11\\
};
\end{axis}
  \end{tikzpicture}}
\qquad\qquad
\subfloat[$\texttt{MODE} = 4$]{
  \begin{tikzpicture}[trim axis right, trim axis left]
    \begin{axis}[%
xmode=log,
xmin=10,
xmax=1000,
xminorticks=true,
xlabel={$n$},
ymode=log,
ymin=1e-16,
ymax=0.0001,
ytick={ 1e-16,  1e-13,  1e-10,  1e-07, 0.0001},
yminorticks=true,
]
\addplot [color=mycolor1, line width=1.0pt, mark size=2.5pt,  only marks, mark=asterisk, mark options={solid, mycolor1}, forget plot]
  table[row sep=crcr]{%
10	6.56314371799685e-16\\
14	1.2329711976341e-15\\
19	9.44487126114361e-16\\
27	1.54739310772636e-15\\
37	2.20913648778757e-15\\
52	2.40575571824835e-15\\
72	4.86233724009213e-15\\
100	6.65206173211406e-15\\
139	8.88349689229487e-15\\
193	1.21290172860247e-14\\
268	1.67923668937796e-14\\
373	2.29681400036926e-14\\
518	3.17164523156901e-14\\
720	3.78085247508975e-14\\
1000	5.09711215965183e-14\\
};
\addplot [color=mycolor2, line width=1.0pt, mark size=2.5pt,  only marks, mark=star, mark options={solid, mycolor2}, forget plot]
  table[row sep=crcr]{%
10	4.59595683547546e-10\\
14	2.19614113885124e-10\\
19	4.44704879062555e-10\\
27	3.459668683205e-10\\
37	2.68476685560778e-10\\
52	9.27225096599596e-11\\
72	3.6652746430453e-10\\
100	2.96625637398607e-10\\
139	1.95170746891949e-09\\
193	2.28431752108675e-10\\
268	9.09663778970172e-10\\
373	1.81039874384922e-09\\
518	1.86350167460764e-09\\
720	5.61841024491945e-10\\
1000	1.71585478861396e-10\\
};
\addplot [color=mycolor3, line width=1.0pt, mark size=2.5pt,  only marks, mark=square, mark options={solid, mycolor3}, forget plot]
  table[row sep=crcr]{%
10	5.77999143189049e-10\\
14	1.43566506521799e-09\\
19	3.89273521271824e-10\\
27	7.12550318596822e-10\\
37	2.13460440938255e-09\\
52	1.26001430668528e-09\\
72	2.04730863852925e-09\\
100	1.0777931768969e-09\\
139	3.155874373967e-09\\
193	1.17395519721995e-09\\
268	9.64361533805373e-10\\
373	8.76095765555229e-10\\
518	8.79181233635951e-10\\
720	1.97199384020411e-09\\
1000	6.29341077248725e-10\\
};
\addplot [color=mycolor4, line width=1.0pt, mark size=2.5pt,  only marks, mark=diamond, mark options={solid, mycolor4}, forget plot]
  table[row sep=crcr]{%
10	3.64098525238241e-10\\
14	4.43317003119976e-10\\
19	7.71398040484352e-10\\
27	1.89164722637693e-10\\
37	4.11877542860131e-10\\
52	1.58433075425038e-10\\
72	1.78687813072885e-09\\
100	1.53931196934677e-09\\
139	1.30224319893603e-09\\
193	1.31380397590025e-10\\
268	3.76381539947161e-10\\
373	1.49699178249403e-10\\
518	2.81422144927098e-10\\
720	3.10396423890027e-09\\
1000	1.93810001969483e-10\\
};
\addplot [color=black, dotted, line width=2.0pt, forget plot]
  table[row sep=crcr]{%
10	2.73183137757246e-13\\
14	6.76906645808488e-13\\
19	7.935841151402e-13\\
27	5.99281296702644e-13\\
37	1.3964169987576e-12\\
52	1.37219617517054e-12\\
72	1.99774188449247e-12\\
100	2.64427538408222e-12\\
139	3.08298932584988e-12\\
193	5.44726848184034e-12\\
268	7.60839025299887e-12\\
373	9.41206972877453e-12\\
518	1.28622850396065e-11\\
720	1.65765152149479e-11\\
1000	2.36981498110001e-11\\
};
\end{axis}
  \end{tikzpicture}}

\subfloat[$\texttt{MODE} = 5$]{
  \begin{tikzpicture}[trim axis right, trim axis left]
    \begin{axis}[%
xmode=log,
xmin=10,
xmax=1000,
xminorticks=true,
xlabel={$n$},
ymode=log,
ymin=1e-16,
ymax=0.0001,
ytick={ 1e-16,  1e-13,  1e-10,  1e-07, 0.0001},
yminorticks=true,
ylabel={$\max_{k}\ferrk$}
]
\addplot [color=mycolor1, line width=1.0pt, mark size=2.5pt,  only marks, mark=asterisk, mark options={solid, mycolor1}, forget plot]
  table[row sep=crcr]{%
10	6.02636529223394e-16\\
14	1.09059358637953e-15\\
19	1.27479952363579e-15\\
27	1.56402834104994e-15\\
37	2.4645391389126e-15\\
52	3.00460730897371e-15\\
72	8.16370565439426e-15\\
100	1.17331858512937e-14\\
139	1.61188178641778e-14\\
193	2.28874135724602e-14\\
268	2.74883366155475e-14\\
373	4.46222064535664e-14\\
518	6.09642456272487e-14\\
720	3.39377387533474e-14\\
1000	4.55077944359613e-14\\
};
\addplot [color=mycolor2, line width=1.0pt, mark size=2.5pt,  only marks, mark=star, mark options={solid, mycolor2}, forget plot]
  table[row sep=crcr]{%
10	2.98251272105922e-15\\
14	2.46805413186255e-12\\
19	1.62997086020947e-12\\
27	2.87158959584625e-11\\
37	1.35791219348193e-11\\
52	1.07488940267249e-10\\
72	1.32560419218881e-10\\
100	2.24047551245281e-10\\
139	1.50591038491036e-10\\
193	2.59211096400085e-10\\
268	4.01039686247144e-10\\
373	3.62692600111058e-10\\
518	3.63319122356886e-10\\
720	5.05821897695285e-10\\
1000	6.22709794846785e-10\\
};
\addplot [color=mycolor3, line width=1.0pt, mark size=2.5pt,  only marks, mark=square, mark options={solid, mycolor3}, forget plot]
  table[row sep=crcr]{%
10	4.20340374329945e-15\\
14	3.80059124099402e-12\\
19	2.793841283049e-12\\
27	1.37204673885001e-10\\
37	4.25259497895152e-11\\
52	1.10345987189287e-10\\
72	1.83086854765145e-10\\
100	1.44556006093101e-10\\
139	2.02473787110473e-10\\
193	3.17228746692137e-10\\
268	2.07727916256771e-10\\
373	3.98760574769805e-10\\
518	5.27845795827792e-10\\
720	4.82775077953914e-10\\
1000	3.34220952888495e-10\\
};
\addplot [color=mycolor4, line width=1.0pt, mark size=2.5pt,  only marks, mark=diamond, mark options={solid, mycolor4}, forget plot]
  table[row sep=crcr]{%
10	4.59627329990532e-15\\
14	2.68058743644754e-12\\
19	3.63236221214081e-12\\
27	9.9970570229917e-12\\
37	1.71629056089885e-11\\
52	4.78490112649139e-11\\
72	1.05533632852233e-10\\
100	6.22390185732577e-11\\
139	5.41735645796338e-10\\
193	3.97591865697236e-10\\
268	2.20380542816885e-10\\
373	4.07704790956311e-10\\
518	3.4096140396769e-10\\
720	3.74462971203914e-10\\
1000	3.61075431180688e-10\\
};
\addplot [color=black, dotted, line width=2.0pt, forget plot]
  table[row sep=crcr]{%
10	1.11038561496304e-14\\
14	1.4315095222124e-14\\
19	1.69294547097749e-14\\
27	2.53449555834905e-13\\
37	5.37421384779909e-14\\
52	1.56229731279242e-13\\
72	1.36113746391512e-12\\
100	1.39471940260537e-12\\
139	3.52700897676148e-12\\
193	3.02087609384579e-12\\
268	6.03705244047972e-12\\
373	1.34443053379004e-11\\
518	1.33060024718733e-11\\
720	2.93632165495658e-11\\
1000	5.17031262796086e-11\\
};
\end{axis}
  \end{tikzpicture}}
\qquad\qquad
\subfloat{
  \begin{tikzpicture}[trim axis right, trim axis left]
    \begin{axis}[
      hide axis,
      xmin=0,
      xmax=1,
      ymin=0,
      ymax=1,
      /tikz/every even column/.append style={column sep=10pt},
      cells={line width=0.9pt, mark size=2.6},
      legend style={at={(0.5,0.75)}, anchor=center, legend columns=1}]
      \addplot[color=mycolor1, line width=1.0pt, mark size=2.5pt, only marks, mark=asterisk, mark options={solid, mycolor1}](-1,0);
      \addlegendentry{MP3JacobiSVD}
      \addplot[color=mycolor2, line width=1.0pt, mark size=2.5pt, only marks, mark=star, mark options={solid, mycolor2}](-1,0);
      \addlegendentry{\inline{DGESVJ}}
      \addplot[color=mycolor3, line width=1.0pt, mark size=2.5pt, only marks, mark=square, mark options={solid, mycolor3}](-1,0);
      \addlegendentry{\inline{DGEJSV}}
      \addplot[color=mycolor4, line width=1.0pt,  mark size=2.5pt, only marks, mark=diamond, mark options={solid, mycolor4}](-1,0);
      \addlegendentry{MATLAB \texttt{svd}}
      \addplot [color=black, dotted, line width=2.0pt](-1,0);
      \addlegendentry{{Reference line}}
    \end{axis}
  \end{tikzpicture}
  }
\caption{Maximum relative forward error $\max_k \ferrk$ for applying MP3JacobiSVD, \inline{DGESVJ}, \inline{DGEJSV} and MATLAB \texttt{svd} to matrices with a fixed number of rows $m = 10^3$, a varying number of columns $n$ from $10$ to $10^3$, a fixed condition number $10^{8}$ and \texttt{MODE} as indicated by the subcaptions.}
\label{fig.fwderr_diff_size}
\end{figure}



In this experiment, we fixed the number of {rows, $m$,} of $A$ to $10^3$ and varied the number of {columns, $n$,} over fifteen logarithmically spaced values between $10$ and $10^3$. We set $\kappa_2(A) = 10^8$.
In Figure~\ref{fig.fwderr_diff_size} we observe that {the errors of} MP3JacobiSVD {lay below the reference line for all test matrices} and it always achieved a smaller value of $\max_k\ferrk$ than the other methods considered. In addition, the rate of growth of $\max_k\ferrk$, with respect to $n$, closely matched our proposed bound $(mn)^{1/2}u\scond(\At)$.

\subsubsection{Special test matrices}
\label{sec.numerical-experiment.special-test}
We executed all four algorithms on the following four test matrices from the Anymatrix Toolbox (Version 5.4.4)~\cite{himi21} to see how $\ferrk$ behaves for each of the individual singular values.
\begin{enumerate}[label=(\arabic*)]
    \item \inline{anymatrix('regtools/blur', 10, 5, 1.4)}: a sparse matrix $A \in \R^{100\times100}$ with $\kappa_2(A) = 10^6$ from the MATLAB package {regtools}~\cite{hans07}.
    \item \inline{anymatrix('gallery/kahan', 50, 1e-2)}: an upper triangular Kahan matrix $A \in \R^{50 \times 50}$ with $\kappa_2(A) \approx 1.6\times 10^{15}$. 
    \item \inline{anymatrix('nessie/whiskycorr')}: a correlation matrix $A \in \R^{86\times 86}$ with $\kappa_2(A) \approx 3 \times 10^{18}$ associated with malt whisky tasting from the NESSIE datasets~\cite{tahi10}.
    \item Gram matrix of \inline{anymatrix('gallery/lauchli', 500, 1e-3)}: the command \inline{anymatrix('gallery/lauchli', 500, 1e-3)} gives a Lauchli matrix $B \in \R^{501\times500}$. The Gram matrix $A = B\tp B$ has condition number $5\times 10^{8}$.
\end{enumerate}

\input{figs/special-test-matrix}

Figure~\ref{fig.special-matrix} shows the typical behavior: for singular values that are about the same size as $\tnorm{A}$, all the algorithms can compute them with high relative accuracy. However, as the singular values get smaller, \inline{DGESVJ}, \inline{DGEJSV}, and MATLAB \texttt{svd} gradually lose their relative accuracy while  our algorithm, MP3JacobiSVD, maintains high relative accuracy even for the smallest singular value.
Figure~\ref{subfig.gallery/kahan} demonstrates
an interesting behavior in which only MATLAB \texttt{svd} lost its relative accuracy. This is because the one-sided Jacobi algorithm can compute singular values with high relative accuracy, provided that the input matrix is well-conditioned after scaling, regardless of the side on which the scaling is applied; see~\cite[sect.~2]{domo04},~\cite[sect.~3]{dges99}, and~\cite{math95}. Here, although $\k(A)$ and $\scond(A)$ are both approximately $10^{15}$, $\scond(A\tp) \approx 10^2$.

\subsection{Timing}
Since all components of MP3JacobiSVD {use optimized kernels provided by MATLAB and LAPACK, together with Advanpix's compiled MEX backend for quadruple precision matrix--matrix multiplication}, we can perform meaningful runtime tests using MATLAB.
We used the \texttt{tic} and \texttt{toc} functions to measure the time elapsed and took an average of two runs. 
The test matrices are generated by \inline{gallery('randsvd')} with a fixed $m = 3000$, $10$ logarithmically spaced values of $n$ from $100$ to $3000$, a fixed  condition number $10^8$ and a set of geometrically distributed singular values. 

\begin{figure}[!tbhp]
\centering
\footnotesize
\subfloat[Total time elapsed\label{subfig.total-time-used}]{
  \begin{tikzpicture}[trim axis right, trim axis left]
    \begin{axis}[%
xmode=log,
xmin=100,
      xmax=3000,
      xtick={100, 250, 500, 1250, 3000},
      xticklabels={$100$, $250$, $500$, $1250$, $3000$},
xminorticks=true,
xlabel={$n$},
ymode=log,
ymin=0.001,
ymax=1000,
yminorticks=true,
ylabel={Runtime (sec)},
]
\addplot [color=mycolor1, line width=1.0pt, mark size=2.5pt,  only marks, mark=asterisk, mark options={solid, mycolor1}, forget plot]
  table[row sep=crcr]{%
100	0.3216633125\\
146	0.520952583333333\\
213	0.996843729166667\\
311	2.39825895833333\\
453	4.80566972916667\\
662	9.119651625\\
965	19.5385897083333\\
1409	44.8563481458333\\
2056	95.9151174375\\
3000	213.3430663125\\
};
\addplot [color=mycolor2, line width=1.0pt, mark size=2.5pt,  only marks, mark=star, mark options={solid, mycolor2}, forget plot]
  table[row sep=crcr]{%
100	0.108093854166667\\
146	0.1748500625\\
213	0.389666020833333\\
311	0.887737895833333\\
453	1.84248545833333\\
662	4.02932616666667\\
965	9.43526595833333\\
1409	22.8807304791667\\
2056	54.5373303333333\\
3000	135.951207416667\\
};
\addplot [color=mycolor3, line width=1.0pt, mark size=2.5pt,  only marks, mark=square, mark options={solid, mycolor3}, forget plot]
  table[row sep=crcr]{%
100	0.0170514166666667\\
146	0.0298368958333333\\
213	0.0575249791666667\\
311	0.141336541666667\\
453	0.221822875\\
662	0.525711979166667\\
965	1.36610960416667\\
1409	3.75051616666667\\
2056	12.6846833958333\\
3000	36.0156751041667\\
};
\addplot [color=mycolor4, line width=1.0pt, mark size=2.5pt,  only marks, mark=diamond, mark options={solid, mycolor4}, forget plot]
  table[row sep=crcr]{%
100	0.004058\\
146	0.00694827083333333\\
213	0.0147837708333333\\
311	0.0262632916666667\\
453	0.0638116875\\
662	0.0632032708333333\\
965	0.1156431875\\
1409	0.197114083333333\\
2056	0.365281645833333\\
3000	0.668854979166667\\
};
\end{axis}
  \end{tikzpicture}}
\qquad\qquad\qquad
\subfloat[Relative runtime of the preconditioning process \label{subfig.relative-timing-for-preconditioning}]{
  \begin{tikzpicture}[trim axis right, trim axis left]
    \begin{axis}[%
xmode=log,
xmin=100,
      xmax=3000,
      xtick={100, 250, 500, 1250, 3000},
      xticklabels={$100$, $250$, $500$, $1250$, $3000$},
xminorticks=true,
xlabel={$n$},
ymin=750,
      ymax=1400,
      ytick={800, 1050, 1300},
      yticklabels={$800$, $1050$, $1300$},
yminorticks=true,
ylabel={Relative runtime},
]
\addplot [color=black, line width=1.0pt, mark size=2.5pt, only marks, mark=triangle, mark options={solid, rotate=90, black}, forget plot]
  table[row sep=crcr]{%
100	842.236089318045\\
146	902.00003975985\\
213	957.846280195353\\
311	1108.44893871387\\
453	1230.61365418748\\
662	1204.17558231226\\
965	1239.53400008039\\
1409	1148.2578397926\\
2056	1111.25125838476\\
3000	1061.11855657474\\
};
\end{axis}
  \end{tikzpicture}}

\subfloat[Runtime for performing the one-sided Jacobi algorithm step\label{subfig.time-used-by-Jacobi}]{
  \begin{tikzpicture}[trim axis right, trim axis left]
    \begin{axis}[%
xmode=log,
xmin=100,
xmax=3000,
xminorticks=true,
      xlabel={$n$},
      xtick={100, 250, 500, 1250, 3000},
      xticklabels={$100$, $250$, $500$, $1250$, $3000$},
ymode=log,
ymin=0.01,
ymax=1000,
yminorticks=true,
ylabel={Runtime (sec)},
]
\addplot [color=mycolor1, line width=1.0pt, mark size=2.5pt, only marks,  mark=asterisk, mark options={solid, mycolor1}, forget plot]
  table[row sep=crcr]{%
100	0.0107208333333333\\
146	0.0205966875\\
213	0.042391375\\
311	0.101529479166667\\
453	0.231776083333333\\
662	0.6427056875\\
965	1.83845754166667\\
1409	5.64583408333333\\
2056	16.6861182708333\\
3000	40.7662731666667\\
};
\addplot [color=mycolor2, line width=1.0pt, mark size=2.5pt, only marks, mark=star, mark options={solid, mycolor2}, forget plot]
  table[row sep=crcr]{%
100	0.108093854166667\\
146	0.1748500625\\
213	0.389666020833333\\
311	0.887737895833333\\
453	1.84248545833333\\
662	4.02932616666667\\
965	9.43526595833333\\
1409	22.8807304791667\\
2056	54.5373303333333\\
3000	135.951207416667\\
};
\end{axis}
  \end{tikzpicture}}
\qquad\qquad\qquad
\subfloat[Total time elapsed with a potential version of MP3JacobiSVD\label{subfig.potential-version-Jacobi}]{
  \begin{tikzpicture}[trim axis right, trim axis left]
    \begin{axis}[%
xmode=log,
xmin=100,
xmax=3000,
  xtick={100, 250, 500, 1250, 3000},
      xticklabels={$100$, $250$, $500$, $1250$, $3000$},
xminorticks=true,
xlabel={$n$},
ymode=log,
ymin=0.01,
ymax=1000,
yminorticks=true,
ylabel={Runtime (sec)}
]
\addplot [color=black, line width=1.0pt, mark size=2.5pt, only marks, mark=o , forget plot]
  table[row sep=crcr]{%
100	0.0654372291666667\\
146	0.100721270833333\\
213	0.180799583333333\\
311	0.394891104166666\\
453	0.811093916666666\\
662	1.53488552083333\\
965	3.59178472916667\\
1409	9.66139691666666\\
2056	25.2476154375\\
3000	59.703179125\\
};
\addplot [color=mycolor2, line width=1.0pt, mark size=2.5pt, only marks, mark=star, mark options={solid, mycolor2}, forget plot]
  table[row sep=crcr]{%
100	0.108093854166667\\
146	0.1748500625\\
213	0.389666020833333\\
311	0.887737895833333\\
453	1.84248545833333\\
662	4.02932616666667\\
965	9.43526595833333\\
1409	22.8807304791667\\
2056	54.5373303333333\\
3000	135.951207416667\\
};
\addplot [color=mycolor3, line width=1.0pt, mark size=2.5pt, only marks, mark=square, mark options={solid, mycolor3}, forget plot]
  table[row sep=crcr]{%
100	0.0170514166666667\\
146	0.0298368958333333\\
213	0.0575249791666667\\
311	0.141336541666667\\
453	0.221822875\\
662	0.525711979166667\\
965	1.36610960416667\\
1409	3.75051616666667\\
2056	12.6846833958333\\
3000	36.0156751041667\\
};
\end{axis}
  \end{tikzpicture}}

\subfloat{
  \begin{tikzpicture}[trim axis right, trim axis left]
    \begin{axis}[
      hide axis,
      xmin=0,
      xmax=1,
      ymin=0,
      ymax=1,
      /tikz/every even column/.append style={column sep=0.3cm},
      cells={line width=0.9pt, mark size=2.6},
      legend style={at={(0.5,0.75)}, anchor=center, legend columns=5}]
      \addplot[color=mycolor1, line width=1.0pt, mark size=2.5pt, only marks, mark=asterisk, mark options={solid, mycolor1}](-1,0);
      \addlegendentry{MP3JacobiSVD}
      \addplot[color=mycolor2, line width=1.0pt, mark size=2.5pt, only marks, mark=star, mark options={solid, mycolor2}](-1,0);
      \addlegendentry{\inline{DGESVJ}}
      \addplot[color=mycolor3, line width=1.0pt, mark size=2.5pt, only marks, mark=square, mark options={solid, mycolor3}](-1,0);
      \addlegendentry{\inline{DGEJSV}}
      \addplot[color=mycolor4, line width=1.0pt,  mark size=2.5pt, only marks, mark=diamond, mark options={solid, mycolor4}](-1,0);
      \addlegendentry{MATLAB \texttt{svd}}
      \addplot [color=black, line width=1.0pt, mark size=2.5pt, only marks, mark=o](-1,0);
      \addlegendentry{MP3JacobiSVD (Potential)}
    \end{axis}
  \end{tikzpicture}
}

\caption{Timing tests for MP3JacobiSVD, \inline{DGESVJ}, \inline{DGEJSV}, MATLAB \texttt{svd}, and a theoretical variant of MP3JacobiSVD in which the preconditioner applied in quadruple precision is only $100$ times slower than in double precision. Figure~\ref{subfig.relative-timing-for-preconditioning} (top-right) shows the time elapsed applying the preconditioner at quadruple precision divided by that at double precision. The test matrices have a fixed number of rows $m = 3000$, a varying number of columns $n$ from $100$ to $3000$, a fixed condition number $10^8$, and  geometrically distributed singular values.}
\label{fig.compare-timing}
\end{figure}
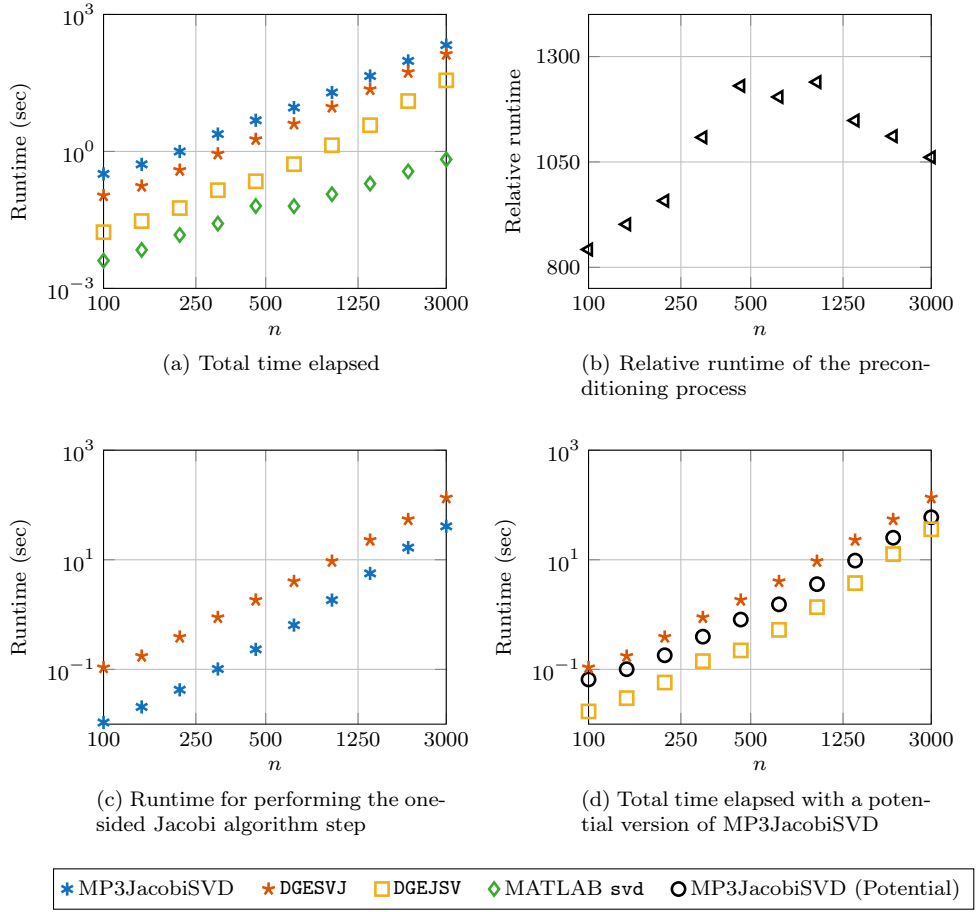


Figure~\ref{subfig.total-time-used} compares the runtimes of the four algorithms and shows that MP3JacobiSVD was the slowest. Nevertheless, we can use MATLAB profiler to assess the time spent on each component of MP3JacobiSVD~\cite[sect.~16.4]{hihi16-MG3}.
The result shows that approximately $66.1\%$ of the time is spent on the matrix--matrix product at quadruple precision, whereas only $23.8\%$ is attributed to the one-sided Jacobi algorithm (\inline{DGESVJ}). The inefficiency of MP3JacobiSVD is expected because quadruple precision was simulated, unlike double- or single-precision arithmetic, which is performed natively on the hardware~\cite{knn23}.

As shown in Figure~\ref{subfig.relative-timing-for-preconditioning}, applying the preconditioner at quadruple precision is over $800$ times slower (but can be up to about $1300$ times slower) than at double precision.
Optimized quadruple precision implementations exist
that are less than 50 times slower,
but these implementations are not available
in MATLAB~\cite[sect.~3.1]{divi06},~\cite{bail05}, and~\cite[sect.~2]{babo15}.
On the other hand, Figure~\ref{subfig.time-used-by-Jacobi} shows that the preconditioning step improves the convergence of the Jacobi algorithm, as {the time used by the SVD stage of MP3JacobiSVD, including the conditional QR factorization and the associated reconstruction of the left singular vectors}, is significantly shorter than the runtime of \inline{DGESVJ}. 

To assess the effect of this computational bottleneck, we simulated the scenario in which the matrix--matrix product is optimized for quadruple precision, so that computing $\Atcomp$ at quadruple precision takes only 100 times as long as at double precision.
The timings in Figure~\ref{subfig.potential-version-Jacobi} suggest that our algorithm could be faster than the \inline{DGESVJ} and could be comparable in runtime with \inline{DGEJSV}.

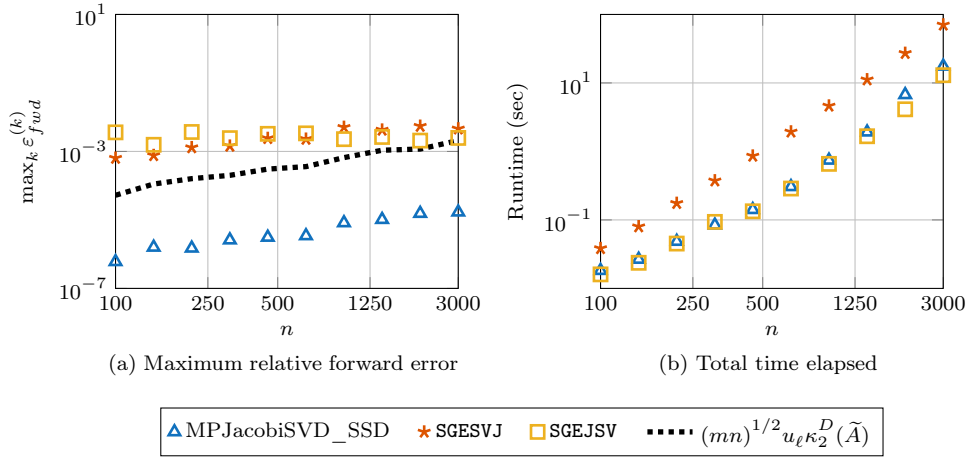
\begin{figure}[!tbhp]
\centering
\footnotesize
\subfloat[Maximum relative forward error\label{subfig.ssd-timing.mferrk}]{
  \begin{tikzpicture}[trim axis right, trim axis left]
    \begin{axis}[%
xmode=log,
xmin=100,
      xmax=3000,
      xtick={100, 250, 500, 1250, 3000},
      xticklabels={$100$, $250$, $500$, $1250$, $3000$},
xminorticks=true,
xlabel={$n$},
ymode=log,
ymin=1e-07,
ymax=10,
yminorticks=true,
ylabel={$\max_{k}\ferrk$},
]
\addplot [color=mycolor1, line width=1.0pt, mark size=2.5pt,  only marks, mark=triangle, mark options={solid, mycolor1}]
  table[row sep=crcr]{%
100	6.01887506945786e-07\\
146	1.58716852638463e-06\\
213	1.4930185443518e-06\\
311	2.58960972132627e-06\\
453	3.1005490654934e-06\\
662	3.43322449225525e-06\\
965	8.21642242954113e-06\\
1409	1.02440262708114e-05\\
2056	1.54203044075985e-05\\
3000	1.69179002114106e-05\\
};

\addplot [color=mycolor2, line width=1.0pt, mark size=2.5pt,  only marks, mark=star, mark options={solid, mycolor2}]
  table[row sep=crcr]{%
100	0.000638420227915049\\
146	0.000766285869758576\\
213	0.00130008778069168\\
311	0.00143913412466645\\
453	0.00241489289328456\\
662	0.00229511712677777\\
965	0.00509154424071312\\
1409	0.00423558987677097\\
2056	0.00548373349010944\\
3000	0.00458662444725633\\
};

\addplot [color=mycolor3, line width=1.0pt, mark size=2.5pt,  only marks, mark=square, mark options={solid, mycolor3}]
  table[row sep=crcr]{%
100	0.00360590382479131\\
146	0.00156265287660062\\
213	0.00369802373461425\\
311	0.00244371616281569\\
453	0.00326625118032098\\
662	0.00336931715719402\\
965	0.0022814276162535\\
1409	0.00264499126933515\\
2056	0.00207266490906477\\
3000	0.00248144124634564\\
};

\addplot [color=black, dotted, line width=2.0pt]
  table[row sep=crcr]{%
100	5.25011746503878e-05\\
146	0.000111578396172263\\
213	0.000160346404300071\\
311	0.000200491151190363\\
453	0.000304989895084873\\
662	0.000360292411642149\\
965	0.000659441866446286\\
1409	0.00108402012847364\\
2056	0.0011743784416467\\
3000	0.0020857525523752\\
};
\end{axis}
  \end{tikzpicture}}
\qquad\qquad\qquad
\subfloat[Total time elapsed\label{subfig.ssd-timing.total-time}]{
  \begin{tikzpicture}[trim axis right, trim axis left]
    \begin{axis}[%
xmode=log,
xmin=100,
      xmax=3000,
      xtick={100, 250, 500, 1250, 3000},
      xticklabels={$100$, $250$, $500$, $1250$, $3000$},
xminorticks=true,
xlabel={$n$},
ymode=log,
ymin=0.01,
ymax=100,
yminorticks=true,
ylabel={Runtime (sec)},
]
\addplot [color=mycolor1, line width=1.0pt, mark size=2.5pt,  only marks, mark=triangle, mark options={solid, mycolor1}]
  table[row sep=crcr]{%
100	0.0182229791666667\\
146	0.026555875\\
213	0.0485236666666667\\
311	0.0839164583333333\\
453	0.141331854166667\\
662	0.305285125\\
965	0.743214416666667\\
1409	1.91544572916667\\
2056	6.72535314583333\\
3000	17.5362966458333\\
};
\addplot [color=mycolor2, line width=1.0pt, mark size=2.5pt,  only marks, mark=star, mark options={solid, mycolor2}]
  table[row sep=crcr]{%
100	0.0383612291666667\\
146	0.0798477708333333\\
213	0.175195166666667\\
311	0.373573645833333\\
453	0.8589139375\\
662	1.93267727083333\\
965	4.6205051875\\
1409	11.1614885208333\\
2056	27.1508228125\\
3000	70.4715142708333\\
};
\addplot [color=mycolor3, line width=1.0pt, mark size=2.5pt,  only marks, mark=square, mark options={solid, mycolor3}]
  table[row sep=crcr]{%
100	0.0160260625\\
146	0.023653625\\
213	0.0451469791666667\\
311	0.0936005833333333\\
453	0.1336679375\\
662	0.2871960625\\
965	0.658285270833333\\
1409	1.66644364583333\\
2056	4.1230380625\\
3000	12.9679499791667\\
};
\end{axis}
  \end{tikzpicture}}

\subfloat{
  \begin{tikzpicture}[trim axis right, trim axis left]
    \begin{axis}[
      hide axis,
      xmin=0,
      xmax=1,
      ymin=0,
      ymax=1,
      /tikz/every even column/.append style={column sep=0.3cm},
      cells={line width=0.9pt, mark size=2.6},
      legend style={at={(0.5,0.75)}, anchor=center, legend columns=5}]
      \addplot[color=mycolor1, line width=1.0pt, mark size=2.5pt,  only marks, mark=triangle, mark options={solid, mycolor1}](-1,0);
      \addlegendentry{MPJacobiSVD\_SSD}
      \addplot[color=mycolor2, line width=1.0pt, mark size=2.5pt,  only marks, mark=star, mark options={solid, mycolor2}](-1,0);
      \addlegendentry{\inline{SGESVJ}}
      \addplot[color=mycolor3, line width=1.0pt, mark size=2.5pt,  only marks, mark=square, mark options={solid, mycolor3}](-1,0);
      \addlegendentry{\inline{SGEJSV}}
      \addplot[color=black, dotted, line width=2.0pt](-1,0);
      \addlegendentry{{Reference line}}
    \end{axis}
  \end{tikzpicture}
}
\vspace{-2.3cm}
\caption{Relative forward accuracy and runtime tests for MPJacobiSVD\_SSD (MP3JacobiSVD with single, single and double precision as low, working and high precision, respectively), \inline{SGESVJ} and \inline{SGEJSV} applied to test matrices with a fixed number of rows $m=3000$, a varying number of columns $n$ from $100$ to $3000$, a fixed condition number of $10^6$, and geometrically distributed singular values.}
\label{fig.compare-timing-ssd}
\end{figure}


Finally, 
{in a separate experiment with $\k(A)=10^6$, we implemented MPJacobiSVD\_SSD}, a version of MP3JacobiSVD with $(\ul, u, \uh) = (\text{single}, \text{single}, \text{double})$, which only uses optimized precisions available on our hardware. {Here, the preconditioner is computed using Algorithm~\ref{alg.orth-method} with orthogonalization step omitted due to $\ul = u$. }
In Figure~\ref{fig.compare-timing-ssd} we observed relative runtimes comparable to the potential relative runtimes in Figure~\ref{subfig.potential-version-Jacobi}. That is, our algorithm is almost as fast as the state-of-the-art preconditioned one-sided Jacobi routine in LAPACK, and significantly faster than the plain one-sided Jacobi. Moreover,
Figure~\ref{subfig.ssd-timing.mferrk} shows that
our algorithm achieved superior accuracy
compared to \inline{SGESVJ} and \inline{SGEJSV} {for all tested matrices}.
This experiment shows the potential of our algorithm:
with further appropriate optimization,
we can obtain accurate singular values without sacrificing efficiency.

\section{Conclusion}\label{sec.conclusion}
Higham et al.~\cite{htwz25} introduced mixed-precision preconditioned Jacobi eigenvalue algorithms that can not only speed up the convergence of the Jacobi iterations, but improve the relative accuracy of the computed eigenvalues. In this paper, we introduced mixed-precision preconditioned one-sided Jacobi algorithms and showed that they achieve analogous improvements for the computation of singular values%
{, with $\scond(A)$ replaced by $\scond(\At)$}.
We introduced two approaches for computing a preconditioner matrix $\Vt$ (Algorithm \ref{alg.orth-method} and Algorithm \ref{alg.low-prec-bidiagonal}). Using a new quantity called the obliquity (see Definition \ref{def.obliq}) we proved conditions under which the high accuracy will be achieved. We believe that the obliquity may find future use in the analysis of SVD algorithms.

By Remark \ref{rem.QR}, performing a QR factorization after our preconditioning process (but not before, as in \cite{gms25}) maintains the high accuracy, and can reduce the dimension of the matrix on which the Jacobi iterations are performed.

Assumption~\ref{it.gammah<u} and
Proposition~\ref{prop:scondlessthan3} suggest that,
{in order to meet the sufficient conditions used in our analysis},
one should set $u_h = c_1 u \kappa_2(A)^{-1}$ and $u_\ell = c_2 \kappa_2(A)^{-1}$ for some tunable parameters $c_1$ and $c_2$.
Hence, if one can estimate the condition number of $A$ (by e.g.~\cite{hiti00, dixo83}) then one can make informed decisions for the precisions.

Figure \ref{subfig.potential-version-Jacobi} and Figure \ref{subfig.ssd-timing.total-time} suggest that, with an optimized implementation of high precision matrix--matrix multiplication, our algorithm can be comparable in speed with the state-of-the-art implementations while attaining much higher accuracy for ill-conditioned matrices, as shown in Figure \ref{fig.fwderr_diff_cond} and Figure \ref{subfig.ssd-timing.mferrk}.

{Algorithm~\ref{alg.mp3jacobi-qr} has
  a natural extension to complex matrices,
  obtained by using the complex one-sided Jacobi
  algorithm~\cite[sects.~3.1 and~4.2]{drma20a}.
  Our numerical experiments show that
  this complex version also achieves
  high relative accuracy for the singular values,
  similar to that observed in the real case.\footnote{%
    See \url{https://github.com/zhengbo0503/Code_twz26a/tree/main/complex_experiments}
    for full performance and accuracy benchmarks
    for the complex version of Algorithm~\ref{alg.mp3jacobi-qr}.
    For brevity, we do not present these numerical experiments
    in this paper.
  }
  We expect that the analysis
  can also be extended to complex arithmetic.
  Indeed, Mathias notes explicitly that
  the error analysis for the Jacobi algorithm
  generalizes to complex matrices~\cite[p.~980]{math95}
  because the basic operations of complex arithmetic
  satisfy rounding-error bounds of the same form
  as the standard model for real arithmetic~\cite[Eq.~2.4]{high02-ASNA2},
  with appropriately enlarged constants~\cite[sect.~3.6]{high02-ASNA2}.
  A complete complex-arithmetic analysis is beyond the scope of this paper.}

\section*{Acknowledgments}
We thank the anonymous referees for their careful reading of our manuscript.
OpenAI GPT-5.6 Sol, accessed through Codex, was used to generate the code
for the complex-arithmetic numerical experiments based on the authors'
original code. The first author reviewed and tested all
generated code. The authors assume responsibility for all content.

\bibliographystyle{siamplain}
\bibliography{bib} 
\end{document}